\documentclass[10pt]{elsarticle}
\usepackage{graphicx}
\usepackage{pifont,latexsym,ifthen,amsthm,rotating,calc,textcase,booktabs}
\usepackage{amsfonts,amssymb,amsbsy,amsmath}
\usepackage{enumitem}
\newtheorem{theorem}{Theorem}
\newtheorem{lemma}{Lemma}
\newtheorem{corollary}{Corollary}
\newtheorem{remark}{Remark}
\newtheorem{proposition}{Proposition}

\newtheorem{dbar}{$\bar{\partial}$-problem}
\usepackage{natbib}
\biboptions{sort&compress}
\newtheorem{rhp}{RH Problem}

\numberwithin{equation}{section}
\usepackage{todonotes}
\usepackage{float}
\usepackage{subfigure}
\usepackage{hyperref}
\hypersetup{
    colorlinks=true, 
    linktoc=all,     
    linkcolor=blue,  
    citecolor=blue
}

\begin{document}

\begin{frontmatter}
\title{Long-time asymptotic behavior of  the Hunter-Saxton equation}

\author[inst2]{Luman Ju}
\author[inst2]{Kai Xu}
\author[inst2]{Engui Fan$^{*,}$  }

\address[inst2]{ School of Mathematical Sciences and Key Laboratory of Mathematics for Nonlinear Science, Fudan University, Shanghai, 200433, China\\
* Corresponding author and e-mail address: faneg@fudan.edu.cn  }

\begin{abstract}

With $\bar{\partial}$-generalization of the Deift-Zhou steepest descent method,
we investigate    the long-time asymptotics of the solution to  the Cauchy problem  for  the Hunter-Saxton (HS) equation
\begin{eqnarray}
&&u_{txx}-2\omega u_x+2u_xu_{xx}+uu_{xxx}=0,\quad x\in \mathbb{R},\ t>0,\label{mhs}\\
&&u(x,0)=u_0(x),  \label{mhs0}
\end{eqnarray}
where $u_0\in H^{3,4}(\mathbb{R})$ and  $\omega>0$ is a constant.
Using the new scale $(y,t)$ and a series of deformations to a Riemann-Hilbert  problem associated with    the Cauchy
  problem, we obtain  the long-time asymptotic approximations of the solution $u(x,t)$ in two  space-time regions:
 The solution  of the HS equation  decays as the speed of $\mathcal{O}(t^{-1/2})$ in the region   $y/t >0$; While
in the region  $y/t<0$,  the solution of the HS equation  is depicted by a parabolic cylinder  model    with an residual error order $\mathcal{O}(t^{-1+\frac{1}{2p}})$
with $ p>2$.
\end{abstract}

\begin{keyword}
 Hunter-Saxton equation, Riemann-Hilbert problem, $\bar{\partial}$-steepest descent method, long-time asymptotics.\\[6pt]

\end{keyword}

\end{frontmatter}
\tableofcontents

\section{Introduction}

This paper is concerned with  the long time asymptotic behavior to the solution of   the Cauchy problem for the Hunter-Saxton (HS) equation \cite{Lnls}
\begin{eqnarray}
&&u_{txx}-2\omega u_x+2u_xu_{xx}+uu_{xxx}=0,\quad x\in \mathbb{R},\ t>0,\label{mhs}\\
&&u(x,0)=u_0(x),  \label{mhs0}
\end{eqnarray}
where $u_0\in H^{3,4}(\mathbb{R})$ and  $\omega>0$ is a constant.  The HS equation \eqref{mhs}   was derived by Hunter and Saxton as a short-wave limit of the Camassa-Holm (CH) equation \cite{HS}
\begin{equation}\label{CH}
u_t-u_{txx}-2\omega u_x+ 3uu_x=2u_xu_{xx}+uu_{xxx},
\end{equation}
which  was discovered as a model for unidirectional propagation of small amplitude shallow water waves by Camassa and Holm \cite{CH}. The CH equation also appears in the study of the propagation of axially
symmetric waves in hyperelastic rods \cite{CH3,CH4}.  Constantin and Lannes further found that the CH equation arises in the modeling of the propagation of shallow water waves over a flat bed  \cite{CH2}.
Due to its astonishing amounts of structures and relating interesting physical phenomena, the CH equation has received extensive attention. For example, Constantin et al. proved the existence and uniqueness of global weak solutions for the CH equation \eqref{CH}   \cite{Cstntin2}, and  discussed the scatting problem of the CH equation  \cite{Cstntin1}.  Besides,
Constantin et al showed  the orbital stability of the peakons and the solitary wave solution for the Cauchy problem of the CH equation \eqref{CH} \cite{Cstntin3,Cstntin4}.
Boutet  de Monvel et al. first investigated the  related Riemann-Hilbert (RH)  problem
	for the CH equation \cite{CH,CH1}.
	They  further  obtained  long-time asymptotics  and  Painlev\'e-type asymptotics for CH equation on the line and  half-line
	by using the  nonlinear steepest descent method	\cite{CH,CH1,Monvel22,Monvel16}. This method developed by  Deift and Zhou \cite{RN6} is
an   effective tool to  rigorously  analyze  the long-time asymptotics behavior of
	  integrable 	systems \cite{RN10,Grunert2009,MonvelCH,xu2015,Geng3,XF2020}.

As a short-wave limit of the CH equation, the HS equation has also received considerable attention.
For instance,
the HS equation (\ref{mhs}) was found to be a model for short capillary waves propagating under the action of gravity \cite{BKMP,FMN}.
Alber et.al studied the weak piecewise solutions of \eqref{mhs} by algebraic geometric methods \cite{Alber}.
The local existence of strong solutions of the periodic Hunter-Saxton equation has been proved by Yin \cite{Yin}.
The multi-cusp solutions of \eqref{mhs} were obtained by Matsuno et.al \cite{Matsu}.
Lenells et al. established Liouville transformation  between the HS hierarchies and the  KdV  hierarchies   \cite{Lens2008}.
Zuo proposed a two-component generalization of the generalized HS equation and  established its  bi-Hamiltonian Euler equation \cite{zuo}.
Zhao constructs the conservation laws of the HS equation for liquid crystal \cite{zhao}.
Lenells provides a rigorous foundation for the geometric interpretation of the HS equation, and yields explicit expressions for the spatially periodic solutions of the related Cauchy problem \cite{Lenells2008}.
Holden et.al developed  an existence theory for the Cauchy problem to the stochastic HS equation and proved  several properties of the blow-up of its solutions \cite{Holden}.
For initial data in Sobolev spaces $H^s(\mathbb{R})$, $s > 3/2$,  Ti\v{g}lay proved  local
well-posedness of   the periodic Cauchy problem for the modified HS equation   \cite{Tiglay}.
Recently Monvel et.al developed the inverse scattering transform and obtained  long-time asymptotics for the HS equation \eqref{mhs} with Schwartz initial data by using
steepest descent method \cite{Monvelmain}.

The aim of our   paper is to derive  the long time asymptotic behavior of the solution   for the HS equations \eqref{mhs}
with  a lower regularity weighted Sobolev  initial data
\begin{equation}
u_0\in H^{3,4}(\mathbb{R})=\{u \in L^2(\mathbb{R})|\langle \cdot \rangle^4 \partial^ju\in L^2(\mathbb{R}),\   j=0,1,\cdots,3\}.
\end{equation}
The key tool to   proved this result      is  the   $\bar{\partial}$-generalization of the steepest descent method, which
 was first proposed by   McLaughlin and Miller \cite{dbar1,dbar2}.
In recent years, this  method shows some advantages and has been successfully used to study long-time asymptotic analysis, soliton resolution and
asymptotic stability of $N$-soliton solutions to integrable systems in a weighted Sobolev space \cite{dbar1,dbar2,dbar3,dbar4,dbar5,dbar6,dbar7,dbar8,dbar9,dbar10,dbar11,HG,dbar12,dbar13,dbar14,dbar15}.

In order to extend the result \cite{Monvelmain} to the solution with  weighted Sobolev initial data  $u_0\in H^{3,4}(\mathbb{R})$  using 
 the $\bar{\partial}$-steepest descent method, there two essential difficulties    need to be overcome.
Firstly, we need to establish a scattering map between   initial data and   reflection coefficient
 $$H^{3,4}(\mathbb{R}) \ni u_0 \to   r(k)\in H^4(\mathbb{R})\cap H^{1,1}(\mathbb{R}),$$
 which is given in  Proposition \ref{rk1}. This require us to make a series of  refined estimates on  eigenfunctions and scattering data starting from
initial data $u_0 \in H^{3,4}(\mathbb{R})$.  Secondly, we need to carefully deal with
the effect of  spectral singularity $k=0$.  In general, for the Schr\"{o}dinger equation and mKdV  equation, their solutions can be
 recovered from  the $k^{-1}$-coefficient of    the solution  of  their related RH problem as  $k\to \infty$ \cite{RN6,dbar7,dbar11}.
However, the solution of the HS equation need to be  recovered  from  the $k^2$-coefficient  of the solution of a  RH problem
as  $k\to 0$.  This  case  results in  a higher singular term  of $k^{-3}$ appearing  in the integral
to estimate    a pure $\bar{\partial}$-problem for $M^{(3)}(k)$.
To avoid  this   singularity,  we construct  a kind of  new   extension  functions $R_{jl}(k)$ such that $|\bar{\partial}R_{jl}|\lesssim |k|^2$ near $k=0$, which is given  in    Proposition \ref{Rjl}.


Without loss of generality, we only discuss the HS equation \eqref{mhs} as $\omega=1$, because under the scaling transformation
\begin{equation}
x=\omega^{-2}\tilde{x},\ t=\omega^{-1}\tilde{t},\ u(x,t)=\omega^{-3}\tilde{u}(\tilde{x},\tilde{t}),\
\end{equation}
the HS equation \eqref{mhs} becomes
\begin{equation*}
\tilde{u}_{\tilde{t}\tilde{x}\tilde{x}}-2 \tilde{u}_{\tilde{x}}+2\tilde{u}_{\tilde{x}}\tilde{u}_{\tilde{x}\tilde{x}}+\tilde{u}\tilde{u}_{\tilde{x}\tilde{x}\tilde{x}}=0.
\end{equation*}
Now we address our main result   as follows.

\begin{theorem}\label{main}
Let $u_0\in H^{3,4}(\mathbb{R})$ with $-u_{0xx} +1 >0,\ for\ all\ x\in \mathbb{R}$. Then the solution  to the Cauchy problem \eqref{mhs}-\eqref{mhs0}
 has the following asymptotic approximation as $t\to\infty$:
\begin{itemize}
  \item when $\xi=\frac{y}{t}>0$, the solution $u(x,t)$ decays fast to $0$,
      \begin{equation}
u(x,t)=\hat{u}(y(x,t),t)=\mathcal{O}(t^{-1/2}),\quad x(y,t)=y+\mathcal{O}(t^{-1/2});
\end{equation}
  \item when $\xi=\frac{y}{t}<0$,  the solution $u(x,t)$  has  the asymptotic expansion
  \begin{equation}
  u(x,t)=\hat{f}t^{-1/2}+\mathcal{O}(t^{-1+\frac{1}{2p}}),
  \end{equation}
  and
  \begin{equation}
x(y,t)=y-\frac{1}{2}\delta_1+\frac{\rho_0^{3/2}}{2}f_{11}t^{-\frac{1}{2}}+\mathcal{O}(t^{-1+\frac{1}{2p}}).
\end{equation}
  where $p>2$, $\hat{f}, f_{11}$ and $\delta_1$ are given in Section 3. Here the variable $y$ is defined as
\begin{equation*}
y=y(x,t)=x-\int_x^\infty(\sqrt{m(s,t)+1}-1)\mathrm{d}s.
\end{equation*}
\end{itemize}
\end{theorem}

This paper is arranged as follows.
In Section \ref{sec2}, with the condition $u_0\in H^{3,4}(\mathbb{R})$, we first make the direct scatting  transform. To improve the strictness of the integral estimates in $\bar{\partial}$-method, we then establish a scattering map between the initial data and the reflection coefficient and prove it detailed.  Further
  we construct a  RH problem associated with  the Cauchy problem \eqref{mhs}-\eqref{mhs0}.
In Section  \ref{sec3}, we deal with the RH problem in the region  $\xi <0$ with two stationary points.
With a series of  transformations,  the original RH problem is deformed into a mixed $\bar{\partial}$-RH problem that
  can be decomposed    into a pure RH problem and a $\bar{\partial}$-problem, which are asymptotically analyzed in
  Subsection \ref{3.3} and   Subsection \ref{3.4}. In this part, to deal with the spectral singularity, we construct a new function with special construction near $k=0$ in Proposition \ref{Rjl}.
 Thus,  we obtain the long-time asymptotic result for the solution $u(x,t)$ of  the HS equation \eqref{mhs} via the reconstruction formula.
Finally, in Section  \ref{sec4}, we show the long-time asymptotics for $u(x,t)$  in the  region  $\xi>0$ in the similar way.

\section{The IST under lower regular initial data }\label{sec2}

In this section,  based on the matrix Lax pair for the HS equation,  we establish
a scattering map  between the  initial data  $u_0\in  H^{3,4}(\mathbb{R})$
and the reflection coefficient $r \in  H^4(\mathbb{R})\cap H^{1,1}(\mathbb{R})$. Further a RH problem associated with  the Cauchy problem \eqref{mhs}-\eqref{mhs0}
   is   constructed.

\subsection{The Lax pair and Jost solutions}
The HS equation \eqref{mhs} admits  following  Lax pair \cite{Monvelmain}
\begin{eqnarray}
&&\tilde{\Phi}_x+\mathrm{i} k \sqrt{m+1} \sigma_3 \tilde{\Phi}=U \tilde{\Phi},\  \label{lax2-1} \\
&&\tilde{\Phi}_t-\mathrm{i} k\left\{\frac{1}{2 k^2}+u \sqrt{m+1}\right\} \sigma_3 \tilde{\Phi}=V \tilde{\Phi},\label{lax2-2}
\end{eqnarray}
where $\tilde{\Phi} \equiv \tilde{\Phi}(x, t, k)$ is a matrix function, $m:=-u_{xx}$, and
\begin{eqnarray}
&&U=\frac{m_x}{4(m+1)} \sigma_1,   \nonumber\\
&&V =-\frac{u m_x}{4(m+1)} \sigma_1-\frac{1}{4 \mathrm{i} k}\left[\sqrt{m+1}+\frac{1}{\sqrt{m+1}}-2\right] \sigma_3.\nonumber
\end{eqnarray}
Here $\sigma_1, \sigma_2, \sigma_3$ are the Pauli matrices
\begin{equation*}
\sigma_1=\left(\begin{array}{ll}
0 & 1 \\
1 & 0
\end{array}\right), \quad \sigma_2=\left(\begin{array}{cc}
0 & -\mathrm{i} \\
\mathrm{i} & 0
\end{array}\right), \quad \sigma_3=\left(\begin{array}{cc}
1 & 0 \\
0 & -1
\end{array}\right).
\end{equation*}
Considering the conservation law of the HS equation, we introduce
\begin{equation}
p(x,t,k)=x-\int_x^\infty(\sqrt{m(s,t)+1}-1)\mathrm{d}s-\frac{t}{2k^2}
\end{equation}
with
\begin{equation*}
p_x=\sqrt{m+1},\quad p_t=-\frac{1}{2k^2}-u\sqrt{m+1}.
\end{equation*}
Then the function $\tilde{\Phi}(x,t,k)$ satisfies the following asymptotics
\begin{equation*}
\tilde{\Phi}(x,t,k)\sim e^{-ikp(x,t,k)\sigma_3},\quad x\to \pm\infty.
\end{equation*}
Define a new matrix valued function
\begin{equation}\label{pp2}
\Phi\equiv\Phi(x,t,k)=\tilde{\Phi}e^{ikp(x,t,k)\sigma_3}
\end{equation}
satisfying the boundary condition
\begin{equation}\label{asymphi}
\Phi\sim I,\quad x\to\pm\infty.
\end{equation}
Following from \eqref{lax2-1}-\eqref{lax2-2}, $\Phi$ admits the new Lax pair
\begin{eqnarray}
&&\Phi_x+ikp_x[\sigma_3,\Phi]=U\Phi,\label{phi_x}\\
&&\Phi_t+ikp_t[\sigma_3,\Phi]=V\Phi,\label{phi_y}
\end{eqnarray}
which allows $\Phi$ to be written as the total differential
\begin{equation}\label{ingrt}
\mathrm{d}(e^{ikp(x,t,k)\hat{\sigma}_3}\Phi)=e^{ikp(x,t,k)\hat{\sigma}_3}[(U\mathrm{d} x+V\mathrm{d} t)\Phi].
\end{equation}
For any fixed $t>0$, integrating \eqref{ingrt} from $\pm\infty$ to $x$, we get
\begin{equation}\label{phipm}
\Phi_{\pm}=I+\int^x_{\pm\infty}e^{-ik\int_y^x\sqrt{m(\tilde{x},t)+1}\mathrm{d}\tilde{x}\hat{\sigma}_3}
(U\Phi_{\pm})\mathrm{d}y.
\end{equation}
On the other hand, the HS equation \eqref{mhs} also admits another matrix Lax pair
\begin{eqnarray}
&&(\tilde{\Phi}_0)_x+ik\sigma\tilde{\Phi}_0=U_0\tilde{\Phi}_0,\label{lax3-1}\\
&&(\tilde{\Phi}_0)_t+\frac{1}{2ik}\sigma_3\tilde{\Phi}_0=V_0\tilde{\Phi}_0,\label{lax3-2}
\end{eqnarray}
where $\tilde{\Phi}_0\equiv\tilde{\Phi}_0(x,t,k)$ is a matrix function and
\begin{eqnarray}
&&U_0=-\frac{ik}{2}m(i\sigma_2+\sigma_3),\nonumber\\
&&V_0=\frac{u_x}{2}\sigma_1+iku[\sigma_3+\frac{m}{2}(i\sigma_2+\sigma_3)].\nonumber
\end{eqnarray}
When $x\to\pm\infty$, $\tilde{\Phi}_0\sim e^{(-ikx-\frac{1}{2ik})\sigma_3}$. Take the transformation on $\tilde{\Phi}_0$ with
\begin{equation}\label{pp4}
\Phi_0\equiv\Phi_0(x,t,k)=\tilde{\Phi}_0e^{(ikx+\frac{t}{2ik})\sigma_3},
\end{equation}
then $\Phi_0\sim I,\ x\to \pm\infty$. And $\Phi_0$ admits the following Lax pair
\begin{eqnarray}
&&(\Phi_0)_x+ik[\sigma_3,\Phi_0]=U_0\Phi_0,\\
&&(\Phi_0)_t+\frac{1}{2ik}[\sigma_3,\Phi_0]=V_0\Phi_0.
\end{eqnarray}
$\Phi_0$ admits the following asymptotic result at $k=0$:
\begin{equation}\label{phi0}
\Phi_0=I+\frac{ik}{2}u_x(i\sigma_2+\sigma_3)-(ik)^2u\sigma_1+O(k^3),\quad k\to 0.
\end{equation}
Considering the linear dependence of $\tilde{\Phi}$ and $\tilde{\Phi}_0$ \cite{Monvelmain}, together with
 equations \eqref{pp2} and \eqref{pp4}, we get the following relation:
\begin{equation}\label{relat}
\Phi_{\pm}=Q(x,t)\Phi_{0\pm}e^{(-ikx-\frac{t}{2ik})\sigma_3}C_{\pm}(k)e^{ikp(x,t,k)\sigma_3},
\end{equation}
where
\begin{equation*}
Q(x,t)=\frac{1}{2}\left(\begin{array}{cc}
q+q^{-1}&q-q^{-1}\\q-q^{-1}&q+q^{-1}\end{array}\right),\quad q=q(x,t)=(m+1)^{\frac{1}{4}},
\end{equation*}
and
\begin{equation}
C_+(k)\equiv I,\quad C_-(k)=e^{ikc\sigma_3},\quad c=\int_{-\infty}^{+\infty}(\sqrt{m+1}-1)\mathrm{d}x. \label{ggd}
\end{equation}

Denote $\Phi_{\pm}$ as
\begin{equation*}
\Phi_{\pm}=(\Phi_{\pm}^{(1)}\quad \Phi_{\pm}^{(2)}),
\end{equation*}
where $\Phi_{\pm}^{(j)},\ j=1,2$ are columns.
By the linear dependence of $\tilde{\Phi}_{\pm}$, we have
\begin{equation}\label{s1}
\Phi_+=\Phi_-e^{-ikp(x,t,k)\hat{\sigma}_3}S(k),
\end{equation}
where
\begin{equation*}
S(k)=\left(\begin{array}{cc}
s_{11}(k)&s_{12}(k)\\
s_{21}(k)&s_{22}(k)\end{array}\right)
\end{equation*}
is independent with $x$ and $t$.

Considering the structure of the Lax pair \eqref{phi_x}-\eqref{phi_y}, $\Phi$ has the following symmetry:
\begin{equation}\label{symt}
\sigma_1\overline{\Phi(x,t,\bar{k})}\sigma_1=\Phi(x,t,k),\quad \Phi(x,t,-k)=\overline{\Phi(x,t,k)},
\end{equation}
thus
\begin{equation*}
s_{11}(k)=\overline{s_{22}(\bar{k})},\quad s_{21}(k)=\overline{s_{12}(\bar{k})}.
\end{equation*}
Therefore, we rewrite $S(k)$ as
\begin{equation}
S(k)=\left(\begin{array}{cc}
\overline{a(\bar{k})}&b(k)\\
\overline{b(\bar{k})}&a(k)\end{array}\right),
\end{equation}
and by \eqref{s1}, $a(k)$ and $b(k)$ are represented as
\begin{equation}\label{ab1}
a(k)=\mid\Phi_-^{(1)}\ \Phi_+^{(2)}\mid,\quad b(k)=e^{2ikp(x,t,k)}\mid\Phi_+^{(2)}\ \Phi_-^{(2)}\mid.
\end{equation}
Based on the asymptotic formula \eqref{asymphi}, we find $|S(k)|\equiv1$ and for $k\in\mathbb{R}$,
\begin{equation*}
|a(k)|^2=1+|b(k)|^2\geqslant1.
\end{equation*}
Combining \eqref{phi0}, \eqref{relat} and \eqref{ab1}, we deduce the expansion of $a(k)$ at $k=0$:
\begin{equation}\label{expa}
a(k)=1+ikc+(ik)^2\frac{c^2}{2}+O(k^3),
\end{equation}
where $c$ is given by (\ref{ggd}).
Besides, $a(k)$ has no zero in $\mathbb{C}$, which is found from the structure of relating Dirac equation \cite{Monvelmain}.

Define the reflection coefficient $r(k)$ as
\begin{equation}\label{rdef}
r(k)=-\frac{\bar{b}(k)}{a(k)},\quad k \in \mathbb{R},
\end{equation}
then
\begin{equation*}
|r(k)|^2=1-|a(k)|^{-2}<1,
\end{equation*}
and it has been proved that for $u_0(x)\in H^{3,4}(\mathbb{R})$, $1-|r(k)|^2>c_0>0$, where $c_0$ is a constant \cite{Monvelmain}. Based on the symmetry \eqref{symt}, we can also derive that $|r(k)|=|r(-k)|$.

With the method of iteration and the Neumann series, we proved that
$
\Phi_-^{(1)},\ \Phi_+^{(2)},\ a(k)$ are\ analytic\ in\ the\ half\ plane\ $\mathbb{C}^+:=\{k\in\mathbb{C}|\mathrm{Im}k>0\}$; $\Phi_-^{(2)},\ \Phi_+^{(1)},\ \overline{a(\bar{k})}$ are\ analytic\ in\ $\mathbb{C}^-:=\{k\in\mathbb{C}|\mathrm{Im}k<0\}$.

\subsection{Scattering map from $u_0(x)$ to  $r(k)$}
In this part, we discuss the relation between the initial data $u_0 $ and the reflection coefficient $r(k)$.  Now we give the main conclusion.
\begin{proposition}\label{rk1}
Suppose the initial data $u_0(x)\in H^{3,4}(\mathbb{R})$ satisfying
$-u_{0xx} +1\geq\epsilon_0>0$, then the map $u_0 \to r(k)$ is Lipschitz continuous from $H^{3,4}(\mathbb{R})$ to $H^4(\mathbb{R})\cap H^{1,1}(\mathbb{R})$.
\end{proposition}
Denote $\Phi_{\pm}(x,k)=\left(\phi_{jl}^{\pm}(x,k)\right)$ as the solutions of \eqref{phipm} for $t=0$. Considering \eqref{symt}, \eqref{ab1} as well as $a(k),\ b(k)$ are independent with $x$ and $t$, taking $x=t=0$, we have
\begin{eqnarray}
&&a(k)=\phi_{11}^-(0,k)\overline{\phi_{11}^+(0,k)}-\phi_{21}^-(0,k)\overline{\phi_{21}^+(0,k)},\nonumber\\
&&e^{2ikc_0}\overline{b(k)}=\phi_{11}^-(0,k)\phi^+_{21}(0,k)-\phi_{21}^-(0,k)\phi^+_{11}(0,k),\nonumber
\end{eqnarray}
where $c_0=\int_0^{\infty}(\sqrt{m+1}-1)\mathrm{d}x$ is real, and thus $\| \overline{b(k)}\|_{L^2(\mathbb{R})}=\| e^{2ikc_0}\overline{b(k)}\|_{L^2(\mathbb{R})}$.

Define $\textbf{n}^{\pm}(x,k)=(n^{\pm}_{11}(x,k),n^{\pm}_{21}(x,k))^T=(\phi_{11}^{\pm}(x,k)-1,\phi_{21}^{\pm}(x,k))^T$, then $a(k),\ b(k)$ are denoted as
\begin{eqnarray}
&&a(k)-1=n_{11}^-(0,k)\overline{n_{11}^+(0,k)}-n_{21}^-(0,k)\overline{n_{21}^+(0,k)}+n_{11}^-(0,k)+\overline{n_{11}^+(0,k)},\qquad\quad\label{an}\\
&&e^{2ikc_0}\overline{b(k)}=n_{11}^-(0,k)n^+_{21}(0,k)-n_{21}^-(0,k)n^+_{11}(0,k)+n^+_{21}(0,k)-n_{21}^-(0,k).\label{bn}
\end{eqnarray}
Therefore, to testify that $r(k)\in H^4(\mathbb{R})$, we only have to prove   $\textbf{n}^{\pm}(0,k)\in H^4(\mathbb{R})$, that is
\begin{proposition}\label{rk2}
The maps
\begin{equation*}
u_0(x)\to n_{11}^{\pm}(0,k),\quad u_0(x)\to n_{21}^{\pm}(0,k)
\end{equation*}
are Lipschitz continuous from $H^{3,4}(\mathbb{R})$ to $H^{4}(\mathbb{R})$.
\end{proposition}
The proof of this Proposition will be detailed later. Now we prove Proposition \ref{rk1} based on the results in Proposition \ref{rk2}.
\begin{proof}(\emph{Proposition \ref{rk1}})

As $\textbf{n}^{\pm}(0,k)\in H^4(\mathbb{R})$, by \eqref{an} and \eqref{bn}, it's obvious that $a(k)$ is bounded and $a'(k),\ a''(k),$ $a'''(k),\ a''''(k)\in L^2(\mathbb{R}),\ b(k)\in H^4(\mathbb{R})$. Thus $r(k)\in H^4(\mathbb{R})$.

We next prove $r(k)\in H^{1,1}(\mathbb{R})$, which equals to prove that $k\overline{b(k)},\ k\overline{b'(k)}\in L^2(\mathbb{R})$. Based on \eqref{phipm}, we find
\begin{equation*}
\begin{aligned}
kn^{\pm}_{21}(0,k)&=k\int_{\pm\infty}^0e^{2i\int_x^0\sqrt{m+1}\mathrm{d}y}\frac{m_x}{4\sqrt{m+1}}\mathrm{d}x+
k\int_{\pm\infty}^0e^{2ik\int_x^0\sqrt{m+1}\mathrm{d}y}\frac{m_x}{4\sqrt{m+1}}n_{11}^{\pm}(x,k)\mathrm{d}x\\
&=-\int_{\pm\infty}^0\frac{1}{8i}\frac{m_x}{m+1}\mathrm{d}e^{2ik\int_x^0\sqrt{m+1}\mathrm{d}y}-\int_{\pm\infty}^0\frac{1}{8i}\frac{m_x}{m+1}n_{11}^{\pm}(x,k)\mathrm{d}e^{2ik\int_x^0\sqrt{m+1}\mathrm{d}y}\\
&=-\frac{1}{8i}\frac{m_x(0)}{m(0)+1}+I_1^{\pm}+I_2^{\pm},
\end{aligned}
\end{equation*}
where
\begin{eqnarray}
&&I_1^{\pm}=-\frac{1}{8i}\frac{m_x(0)}{m(0)+1}n_{11}^{\pm}(0,k),\nonumber\\
&&I_2^{\pm}=\int_{\pm\infty}^0\frac{1}{8i}[\frac{m_x}{m+1}(1+n_{11}^{\pm}(x,k))]_xe^{2ik\int_x^0\sqrt{m+1}\mathrm{d}y}\mathrm{d}x\nonumber
\end{eqnarray}
belong to $L^2(\mathbb{R})$.
Therefore, by \eqref{bn}, we have
\begin{equation*}
\begin{aligned}
e^{2ikc_0}k\overline{b(k)}=&-\frac{1}{8i}\frac{m_x(0)}{m(0)+1}(n^-_{11}(0,k)-n^+_{11}(0,k))+(I_1^{+}+I_2^{+})n^-_{11}(0,k)\\
&-(I_1^{-}+I_2^{-})n^+_{11}(0,k)+(I_1^{+}+I_2^{+})-(I_1^{-}+I_2^{-}).
\end{aligned}
\end{equation*}
Thus we conclude that $k\overline{b(k)}\in L^2(\mathbb{R})$, and the proof of $k\overline{b'(k)} \in L^2(\mathbb{R})$ is similar.
\end{proof}
Therefore, we only have to prove Proposition \ref{rk2}, which needs some preparation work with several lemmas.
Take $\textbf{n}^+(x,k)$ for example and to convenient, we replace it by $\textbf{n}(x,k)$. By \eqref{phipm}, we have
\begin{equation}\label{n}
\textbf{n}(x,k)=\textbf{n}_0(x,k)+T\textbf{n}(x,k),
\end{equation}
where $T$ is an integral operator defined by
\begin{equation}
T\textbf{f}(x,k)=\int_x^{+\infty}K(x,y,k)\textbf{f}(y,k)\mathrm{d}y,
\end{equation}
with the kernel
\begin{equation}\label{K}
K(x,y,k)=\left(\begin{array}{cc}0&-\frac{m_y}{4(m+1)}\\-\frac{m_y}{4(m+1)}e^{2ik(h(y)-h(x))}&0\end{array}\right),
\end{equation}
and
\begin{equation}
\textbf{n}_0(x,k)=T\textbf{e}_1=\left(\begin{array}{c}0\\ -\int_x^{+\infty}\frac{m_y}{4(m+1)}e^{2ik(h(y)-h(x))}\mathrm{d}y\end{array}\right).
\end{equation}
Here the function $h(x)$ is defined as
\begin{equation*}
h(x)=\int_x^{\infty}\sqrt{m+1}\mathrm{d}\zeta,
\end{equation*}
and thus
\begin{equation*}
h(x)-h(y)=\int_x^y\sqrt{m+1}\mathrm{d}\zeta.
\end{equation*}
Taking the partial derivatives of $k$ for \eqref{n}, we get
\begin{eqnarray}
&&(\textbf{n})_k=\textbf{n}_1+T(\textbf{n})_k,\quad  \textbf{n}_1=(\textbf{n}_0)_k+(T)_k\textbf{n},\label{nk}\\
&&(\textbf{n})_{kk}=\textbf{n}_2+T(\textbf{n})_{kk},\quad \textbf{n}_2=(\textbf{n}_0)_{kk}+(T)_{kk}\textbf{n}+2(T)_k(\textbf{n})_k,\label{nkk}\\
&&(\textbf{n})_{kkk}=\textbf{n}_3+T(\textbf{n})_{kkk},\ \textbf{n}_3=(\textbf{n}_0)_{kkk}+(T)_{kkk}\textbf{n}+3(T)_{kk}(\textbf{n})_k+3(T)_k(\textbf{n})_{kk},\qquad\label{nkkk}\\
&&(\textbf{n})_{kkkk}=\textbf{n}_4+T(\textbf{n})_{kkkk},\ \textbf{n}_4=(\textbf{n}_0)_{kkkk}+(T)_{kkkk}\textbf{n}+4(T)_{kkk}(\textbf{n})_k+6(T)_{kk}(\textbf{n})_{kk}+4(T)_k(\textbf{n})_{kkk}.\quad\qquad\label{nkkkk}
\end{eqnarray}
To find the solutions of the differential equations \eqref{n}, \eqref{nk}, \eqref{nkk}, \eqref{nkkk} and \eqref{nkkkk}, we need several lemmas as follows:
\begin{lemma}\label{lem1}
For $u_0 \in H^{3,4}(\mathbb{R})$, the following estimates hold:
\begin{equation}\label{n0est}
\|\textbf{n}_0\|_{C^0(\mathbb{R}^+,L^2(\mathbb{R}))}\lesssim \| m_x\|_{L^2},\quad \| \textbf{n}_0\|_{L^2(\mathbb{R}^+\times\mathbb{R})}\lesssim \| m_x\|_{L^{2,\frac{1}{2}}};
\end{equation}
\begin{equation}\label{n0exp}
\begin{aligned}
&\| (\textbf{n}_0)_k\|_{C^0(\mathbb{R}^+,L^2(\mathbb{R}))}\lesssim \| m_x\|_{L^{2,1}}+\| m\|_{L^1}\| m_x\|_{L^2},\\ &\| (\textbf{n}_0)_k\|_{L^2(\mathbb{R}^+\times\mathbb{R})}\lesssim \| m_x\|_{L^{2,\frac{3}{2}}}+\| m\|_{L^1}\| m_x\|_{L^{2,\frac{1}{2}}};\\
\end{aligned}
\end{equation}
\begin{equation}
\begin{aligned}
&\| (\textbf{n}_0)_{kk}\|_{C^0(\mathbb{R}^+,L^2(\mathbb{R}))}\lesssim \| m_x\|_{L^{2,2}}+\| m\|_{L^1}\| m_x\|_{L^{2,1}}+\| m\| ^2_{L^1}\| m_x\|_{L^2},\\
&\| (\textbf{n}_0)_{kk}\|_{L^2(\mathbb{R}^+\times\mathbb{R})}\lesssim \| m_x\|_{L^{2,\frac{5}{2}}}+\| m\|_{L^1}\| m_x\|_{L^{2,\frac{3}{2}}}+\| m\| ^2_{L^1}\| m_x\|_{L^{2,\frac{1}{2}}};
\end{aligned}
\end{equation}
\begin{equation}
\begin{aligned}
&\|(\textbf{n}_0)_{kkk}\|_{C^0(\mathbb{R}^+,L^2(\mathbb{R}))}\lesssim\| m_x\|_{L^{2,3}}+\| m\|_{L^1}\| m_x\|_{L^{2,2}}+\| m\|^2_{L^1}\|m_x\|_{L^{2,1}}+\| m\| ^3_{L^1}\| m_x\|_{L^2},\\
&\| (\textbf{n}_0)_{kkk}\|_{L^2(\mathbb{R}^+\times\mathbb{R})}\lesssim\| m_x\|_{L^{2,\frac{7}{2}}}+\| m\|_{L^1}\| m_x\|_{L^{2,\frac{5}{2}}}+\| m\| ^2_{L^1}\| m_x\|_{L^{2,\frac{3}{2}}}+\| m\| ^3_{L^1}\| m_x\|_{L^{2,\frac{1}{2}}};
\end{aligned}
\end{equation}
\begin{equation}
\begin{aligned}
\| (\textbf{n}_0)_{kkkk}\|_{C^0(\mathbb{R}^+,L^2(\mathbb{R}))}&\lesssim \| m_x\|_{L^{2,4}}+\| m\|_{L^1}\| m_x\|_{L^{2,3}}+\| m\| ^2_{L^1}\| m_x\|_{L^{2,2}}+\| m\| ^3_{L^1}\| m_x\|_{L^{2,1}}\\
&+\| m\| ^4_{L^1}\| m_x\|_{L^{2}}.
\end{aligned}
\end{equation}
\end{lemma}
\begin{proof}
We take the proof of \eqref{n0exp} for example, and the rest are deduced similarly.

Take the derivative with respect to $\textbf{n}_0(x,k)$ for $k$, we get
\begin{equation*}
(\textbf{n}_0)_k(x,k)=\left(\begin{array}{c}0\\ -2i(h(y)-h(x))\int_x^{+\infty}\frac{m_y}{4(m+1)}e^{2ik(h(y)-h(x))}\mathrm{d}y\end{array}\right).
\end{equation*}
Considering that for $y>x$,
\begin{equation*}
h(x)-h(y)=\int_x^y\sqrt{m+1}\mathrm{d}\zeta\lesssim(y-x)+\| m\|_{L^1},
\end{equation*}
we deduce that for any function $\varphi(k)\in L^2(\mathbb{R})$ satisfying $\| \varphi\|_{L^2}=1$,
\begin{equation*}
\begin{aligned}
&\| (\textbf{n}_0)_k\|_{L^2(\mathbb{R})}=\sup_{\varphi\in L^2(\mathbb{R})}\int_0^{\infty}2i(h(x)-h(y))\int_x^{+\infty}\frac{m_y}{4(m+1)}e^{2ik(h(y)-h(x))}\varphi(k)\mathrm{d}y\mathrm{d}k\\
\lesssim&\sup_{\varphi\in L^2(\mathbb{R})}\left(\int_x^{+\infty}\frac{(y-x)m_y}{4(m+1)}\widehat{\varphi}(h(x)-h(y))\mathrm{d}y
+\| m\|_{L^1}\int_x^{+\infty}\frac{m_y}{4(m+1)}\widehat{\varphi}(h(x)-h(y))\mathrm{d}y\right)\\
\lesssim&\left(\int_x^{+\infty}|ym_y|^2\mathrm{d}y\right)^{1/2}
+\| m\|_{L^1}\left(\int_x^{+\infty}|m_y|^2\mathrm{d}y\right)^{1/2}.
\end{aligned}
\end{equation*}
Therefore,
\begin{equation*}
\| (\textbf{n}_0)_k\|_{C^0(\mathbb{R}^+,L^2(\mathbb{R}))}=\sup_{x\geqslant0}\| (\textbf{n}_0)_k\|_{L^2(\mathbb{R})}\lesssim\| m_x\|_{L^{2,1}}+\| m\|_{L^1}\| m_x\|_{L^2},
\end{equation*}
and
\begin{equation*}
\begin{aligned}
\| (\textbf{n}_0)_k\|_{L^2(\mathbb{R}^+\times\mathbb{R})}&\lesssim\left(\int_0^{+\infty}\int_x^{+\infty}|ym_y|^2\mathrm{d}y\mathrm{d}x
+\| m\|_{L^1}\int_0^{+\infty}\int_x^{+\infty}|m_y|^2\mathrm{d}y\mathrm{d}x\right)^{1/2}\\
&\lesssim\left(\int_0^{+\infty}\int_0^y|ym_y|^2\mathrm{d}x\mathrm{d}y\right)^{1/2}+\| m\|_{L^1}\left(\int_0^{+\infty}\int_0^y|m_y|^2\mathrm{d}x\mathrm{d}y\right)^{1/2}\\
&\lesssim\| m_x\|_{L^{2,\frac{3}{2}}}+\| m\|_{L^1}\| m_x\|_{L^{2,\frac{1}{2}}}.
\end{aligned}
\end{equation*}

\end{proof}
Next, we deal with the operators $(T)_k,\ (T)_{kk}$ and $(T)_{kkk}$, which have the integral kernel $(K)_k,\ (K)_{kk}$ and $(K)_{kkk}$ separately, where
\begin{equation}
(K)_{k}(x,y,k)=\left(\begin{array}{cc}0&0\\2i(h(x)-h(y))\frac{m_y}{4(m+1)}e^{2ik(h(y)-h(x))}&0\end{array}\right).
\end{equation}
$(K)_{kk}$,\ $(K)_{kkk}$ and $(K)_{kkkk}$ have the totally same form  with $2i(h(x)-h(y))$ replaced by $(2i(h(x)-h(y)))^2$,\ $(2i(h(x)-h(y)))^3$ and $(2i(h(x)-h(y)))^4$. These operators admit following estimates:
\begin{lemma}\label{tk}
For $u_0 \in H^{3,4}(\mathbb{R})$, the following operator bounds hold uniformly, and the operators are Lipschitz continuous of $u_0(x)$.
\begin{equation*}
\begin{aligned} &\| (T)_k\|_{L^2(\mathbb{R}^+\times\mathbb{R})\to C^0(\mathbb{R}^+,L^2(\mathbb{R}))}\lesssim\| m_x\|_{L^{2,1}}+\| m\|_{L^1}\| m_x\|_{L^2},\\
&\| (T)_k\|_{L^2(\mathbb{R}^+\times\mathbb{R})\to L^2(\mathbb{R}^+\times\mathbb{R})}\lesssim\| m_x\|_{L^{2,\frac{3}{2}}}+\| m\|_{L^1}\| m\|_{L^{2,\frac{1}{2}}};
\end{aligned}
\end{equation*}
\begin{equation*}
\begin{aligned} &\| (T)_{kk}\|_{L^2(\mathbb{R}^+\times\mathbb{R})\to C^0(\mathbb{R}^+,L^2(\mathbb{R}))}\lesssim\| m_x\|_{L^{2,2}}+\| m\|_{L^1}\| m_x\|_{L^{2,1}}+\| m\| ^2_{L^1}\| m_x\|_{L^2},\\
&\| (T)_{kk}\|_{L^2(\mathbb{R}^+\times\mathbb{R})\to L^2(\mathbb{R}^+\times\mathbb{R})}\lesssim\| m_x\|_{L^{2,\frac{5}{2}}}+\| m\|_{L^1}\| m\|_{L^{2,\frac{3}{2}}}+\| m\| ^2_{L^1}\| m_x\|_{L^{2,\frac{1}{2}}};
\end{aligned}
\end{equation*}
\begin{equation*}
\begin{aligned}
&\| (T)_{kkk}\|_{L^2(\mathbb{R}^+\times\mathbb{R})\to C^0(\mathbb{R}^+,L^2(\mathbb{R}))}\lesssim\| m_x\|_{L^{2,3}}+\| m\|_{L^1}\| m_x\|_{L^{2,2}}+\| m\| ^2_{L^1}\| m_x\|_{L^{2,1}}+\| m\| ^3_{L^1}\| m_x\|_{L^2},\\
&\| (T)_{kkk}\|_{L^2(\mathbb{R}^+\times\mathbb{R})\to L^2(\mathbb{R}^+\times\mathbb{R})}\lesssim\| m_x\|_{L^{2,\frac{7}{2}}}+\| m\|_{L^1}\| m\|_{L^{2,\frac{5}{2}}}+\| m\| ^2_{L^1}\| m_x\|_{L^{2,\frac{3}{2}}}+\| m\| ^3_{L^1}\| m_x\|_{L^{2,\frac{1}{2}}};
\end{aligned}
\end{equation*}
\begin{align*}
&\| (T)_{kkkk}\|_{L^2(\mathbb{R}^+\times\mathbb{R})\to C^0(\mathbb{R}^+,L^2(\mathbb{R}))}\lesssim\| m_x\|_{L^{2,4}}
+\| m\|_{L^1}\| m_x\|_{L^{2,3}}+\| m\| ^2_{L^1}\| m_x\|_{L^{2,2}}\\
&+\| m\| ^3_{L^1}\| m_x\|_{L^{2,1}}+\| m\| ^4_{L^1}\| m_x\|_{L^2}.
\end{align*}
\end{lemma}
The proof of this lemma is the same as Lemma \ref{lem1}.

To solve the equations \eqref{n}, \eqref{nk}, \eqref{nkk}, \eqref{nkkk} and \eqref{nkkkk}, we finally discuss the existence of the operator $(I-T)^{-1}$. Denote $f^*(x)=\sup_{y\geqslant x}\| f(y,\cdot)\|_{L^2(\mathbb{R})}$, then by \eqref{K}, we find $K(x,y,k)\leqslant g(y)$ and
\begin{equation}\label{Tf*}
(Tf)^*(x)\leqslant\int_x^{\infty}g(y)f^*(y)\mathrm{d}y,
\end{equation}
where
\begin{equation*}
g(y)=\frac{m_y}{2(m+1)}.
\end{equation*}

Therefore, the resolvent $(I-T)^{-1}$ exists with following lemma:
\begin{lemma}\label{revol}
For each $k\in\mathbb{R}$ and $u_0 \in H^{3,4}(\mathbb{R})$, $(I-T)^{-1}$ exists as a bounded operator from $C^0(\mathbb{R}^+)$ to itself. What's more, $\hat{L}:=(I-T)^{-1}-I$ is an integral operator with continuous integral kernel $L(x,y,k)$ satisfying
\begin{equation}\label{gn}
|L(x,y,k)|\leqslant \exp(\| g\|_{L^1})g(y).
\end{equation}
\end{lemma}
\begin{proof}
By \eqref{n}, it's obvious that $T$ is a Volterra operator, and together with \eqref{Tf*}, we can deduce that $(I-T)^{-1}$ exists unique as a bounded operator on $C^0(\mathbb{R}^+)$. For the operator $\hat{L}$, the integral kernel $L(x,y,k)$ is given by
\begin{equation*}
L(x,y,k)=\left\{\begin{array}{ll}
\sum_{n=1}^{\infty}K_n(x,y,k),&x\leqslant y,\\
0,&x>y,
\end{array}\right.
\end{equation*}
where
\begin{equation*}
K_n(x,y,k)=\int_{x\leqslant y_1\leqslant\cdots\leqslant y_{n-1}}K_(x,y_1,k)K(y_1,y_2,k)\cdots K(y_{n-1},y,k)\mathrm{d}y_{n-1}\cdots \mathrm{d}y_1.
\end{equation*}
By the estimate $|K(x,y,k)|\leqslant g(y)$, we get
\begin{equation*}
|K_n(x,y,k)|\leqslant\frac{1}{(n-1)!}\left(\int_x^{\infty}g(t)\mathrm{d}t\right)^{n-1}g(y),
\end{equation*}
and then \eqref{gn} follows.
\end{proof}

By \eqref{Tf*}, we find that $T$ is a bounded operator as $T:L^2\to C^0,\ T:C^0\to L^2,$ and $T:L^2\to T^2$. Therefore, by the formula
\begin{equation*}
\hat{L}=(I-T)^{-1}-I=T+T(I-T)^{-1}T,
\end{equation*}
we deduce that $\hat{L}$ is a bounded operator as
$\hat{L}:C^0(\mathbb{R}^+,L^2(\mathbb{R}))\to C^0(\mathbb{R}^+,L^2(\mathbb{R}))$ and $\hat{L}:L^2(\mathbb{R}^+\times\mathbb{R})\to L^2(\mathbb{R}^+\times\mathbb{R})$.

Based on above results, we now turn back to give the proof of Proposition \ref{rk2}.
\begin{proof}(\emph{Proposition \ref{rk2}})

By \eqref{n}, we find
\begin{equation}\label{n2}
  \textbf{n}(x,k)=((I-T)^{-1}-I)\textbf{n}_0(x,k)+\textbf{n}_0(x,k).
\end{equation}
By \eqref{n0est} in Lemma \ref{lem1}, $\textbf{n}_0(x,k)\in C^0(\mathbb{R}^+,L^2(\mathbb{R}))\cap L^2(\mathbb{R}^+\times\mathbb{R})$,
and then Lemma \ref{revol} guarantees that there exists unique solution $\textbf{n}(x,k)$ of \eqref{n2} with $\textbf{n}(x,k)\in C^0(\mathbb{R}^+,L^2(\mathbb{R}))\cap L^2(\mathbb{R}^+\times\mathbb{R})$.
Similarly, together with Lemma \ref{tk}, we have
\begin{equation*}
\begin{aligned}
&\textbf{n}_k(x,k)\in C^0(\mathbb{R}^+,L^2(\mathbb{R}))\cap L^2(\mathbb{R}^+\times\mathbb{R}),\quad \textbf{n}_{kk}(x,k)\in C^0(\mathbb{R}^+,L^2(\mathbb{R}))\cap L^2(\mathbb{R}^+\times\mathbb{R}),\\
&\textbf{n}_{kkk}(x,k)\in C^0(\mathbb{R}^+,L^2(\mathbb{R}))\cap L^2(\mathbb{R}^+\times\mathbb{R}),\quad\textbf{n}_{kkkk}(x,k)\in C^0(\mathbb{R}^+,L^2(\mathbb{R})).
\end{aligned}
\end{equation*}
Taking $x=0$ in all above, we get $\textbf{n}(0,k)\in H^4(\mathbb{R})$.
\end{proof}
At the end of this part, we give a remark as a supplement of Proposition \ref{rk1}. It plays an important role in solving the singularity at $k=0$ in following sections.
\begin{remark}\label{rmk1}
If  $r(k)\in H^4(\mathbb{R})$, then  $r(k)\in C^{3}(\mathbb{R})$  by the Sobolev  embedding theorem.
\end{remark}

\subsection{The Riemann-Hilbert problem of the HS equation}
Define a new matrix-valued function $M(x,t,k)\equiv M(k)$ as follows:
\begin{equation}
M(k)=\left\{\begin{array}{lr}
\left(\frac{\Phi_-^{(1)}}{a\left(k\right)}\quad \Phi_+^{(2)}\right),\quad \mathrm{Im}k>0,\\
\left(\Phi_+^{(1)}\quad \frac{\Phi_-^{(2)}}{\overline{a(\bar{k})}}\right),\quad \mathrm{Im}k<0,\end{array}\right.
\end{equation}
which solves the following RH problem
\begin{rhp}\label{rhp1}
Find a $2\times 2$ matrix-valued function $M(k)$ satisfying \begin{itemize}
             \item Analyticity: $M(k)$ is analytic in $\mathbb{C}\setminus\mathbb{R};$
             \item Symmetry: $M(k)=\sigma_1 \overline{ M(\bar{k})}\sigma_1$,\quad $M(-k)=\overline{M(k)};$
             \item Jump condition: $M(k)$ has continuous boundary values $M_{\pm}(k)$ on $\mathbb{R}$ and
\begin{equation}
M_+(k)=M_-(k)V(x,t,k),\quad k\in \mathbb{R},
\end{equation}
where
\begin{equation}
V(x,t,k)=e^{-ikp(x,t,k)\hat{\sigma}_3}V_0(k),
\end{equation}
and
\begin{equation}
V_0(k)=\left(\begin{array}{cc}
1-|r(k)|^2&-\bar{r}(k)\\r(k)&1\end{array}\right);
\end{equation}
             \item Asymptotic behaviors:
\begin{equation}
M(k)=I+\mathcal{O}(k^{-1}),\quad k\to \infty.
\end{equation}
           \end{itemize}
\end{rhp}
Denoting $M(k)=(M_{lj}(k)),\ M_j(k)=M_{1j}(k)+M_{2j}(k)$,
from the expansion of $M_j(k)$ at $k=0$:
\begin{equation*}
M_j(k)\triangleq M_j(0)+M_{j,1}k+M_{j,2}k^2+\mathcal{O}(k^3),\ j=1,2,
\end{equation*}
we find
\begin{equation}\label{restr1}
\begin{aligned}
M_1(k)M_2(k)=\sqrt{m+1}(1+2k^2u+\mathcal{O}(k^3)).
\end{aligned}
\end{equation}

Introduce a new scale $(y,t)$ by
\begin{equation}\label{ytx}
y\equiv y(x,t)=x-\int_x^\infty(\sqrt{m(\zeta,t)+1}-1)\mathrm{d}\zeta,
\end{equation}
then we give a new matrix-valued function as
\begin{equation}\label{xty}
\widehat{M}(k)\equiv\widehat{M}(y(x,t),t,k):=M(x,t,k).
\end{equation}
By the coordinate transformation \eqref{ytx}, $\widehat{M}(k)$ obeys the jump condition
\begin{equation}
\widehat{M}_+(k)=\widehat{M}_-(k)\widehat{V}(y,t,k),
\end{equation}
where
\begin{equation}
\quad \widehat{V}(y(x,t),t,k)=V(x,t,k)=e^{-it\theta(\xi,k)\hat{\sigma}_3}V_0(k),
\end{equation}
and
\begin{equation}\label{theta}
\theta(\xi,k)=k\xi-\frac{1}{2k},\quad \xi=\frac{y}{t}.
\end{equation}
Combining \eqref{restr1} and \eqref{xty}, we reconstruct the formula of $u(x,t)$ as follows
\begin{equation}\label{rstr-u}
\begin{aligned}
u(x,t)&=\hat{u}(y(x,t),t)
=\lim_{k\rightarrow0}\frac{1}{2k^2}\left(
\frac{\widehat{M}_1(y,t,k)\widehat{M}_2(y,t,k)}
{\widehat{M}_1(y,t,0)\widehat{M}_2(y,t,0)}-1\right)\\
&=\widehat{M}_{1,1}\widehat{M}_{2,1}(y(x,t),t)+\widehat{M}_{1,2}(y(x,t),t)+\widehat{M}_{2,2}(y(x,t),t),
\end{aligned}
\end{equation}
where
\begin{equation}
x(y,t)=y-\lim_{k\rightarrow0}\frac{1}{2ik}\left(
\frac{\widehat{M}_1(y,t,k)}{\widehat{M}_1(y,t,0)}-1\right),
\end{equation}
and
\begin{equation*}
\widehat{M}_{j,l}(y(x,t),t):=M_{j,l}(x,t),\ j,l=1,2.
\end{equation*}

By \eqref{theta},
\begin{equation} \label{Imthta} \mathrm{Im}\theta=\mathrm{Im}k(\xi+\frac{1}{2|k|^2}),
\end{equation}
and thus we have the following lemma from direct calculation:
\begin{lemma}\label{Imk}
For $\xi\neq0$,\quad
$
\mathrm{Im} \theta(k)>\xi \mathrm{Im}k,\  as\ \mathrm{Im}k>0;\quad
\mathrm{Im} \theta(k)<\xi \mathrm{Im}k,\  as\ \mathrm{Im}k<0.
$
\end{lemma}
Before the asymptotic analysis, we first divide the $(y,t)$ half-plane into two regions as shown in Figure \ref{region} :

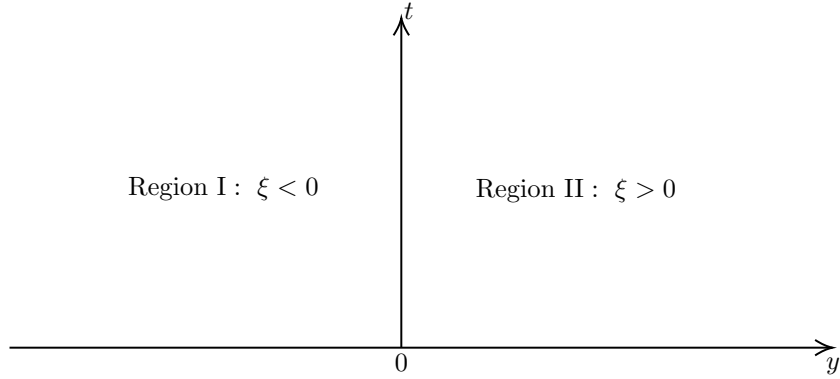
\begin{figure}
\tikzset{every picture/.style={line width=0.75pt}} 
\begin{tikzpicture}[x=0.75pt,y=0.75pt,yscale=-0.8,xscale=0.8]
uncomment if require: \path (-10,100); 

\draw    (71,254) -- (587.8,254) ;
\draw [shift={(589.8,254)}, rotate = 180] [color={rgb, 255:red, 0; green, 0; blue, 0 }  ][line width=0.75]    (10.93,-4.9) .. controls (6.95,-2.3) and (3.31,-0.67) .. (0,0) .. controls (3.31,0.67) and (6.95,2.3) .. (10.93,4.9)   ;
\draw    (318,48) -- (318,197) -- (318,254) ;
\draw [shift={(318,46)}, rotate = 90] [color={rgb, 255:red, 0; green, 0; blue, 0 }  ][line width=0.75]    (10.93,-4.9) .. controls (6.95,-2.3) and (3.31,-0.67) .. (0,0) .. controls (3.31,0.67) and (6.95,2.3) .. (10.93,4.9)   ;

\draw (584,258.4) node [anchor=north west][inner sep=0.75pt]    {$y$};
\draw (318,34.4) node [anchor=north west][inner sep=0.75pt]    {$t$};
\draw (312,256.4) node [anchor=north west][inner sep=0.75pt]    {$0$};
\draw (144,144.4) node [anchor=north west][inner sep=0.75pt]    {$\mathrm{Region} \ \mathrm{I} :\ \xi < 0$};
\draw (363,145.4) node [anchor=north west][inner sep=0.75pt]    {$\mathrm{Region} \ \mathrm{II} :\ \xi  >0$};
\end{tikzpicture}
\caption{The different regions of the $(y,t)$ half-plane, $\xi=\frac{y}{t}$.}
\label{region}
\end{figure}
\begin{enumerate}
    \item Region I: For $\xi<0$, the sign of $\mathrm{Im}\theta$ is shown in Figure \ref{sign1}, where $k_j$ are the two stationary points:
      \begin{equation}
      k_1=-\sqrt{-\frac{1}{2\xi}},\quad k_2=\sqrt{-\frac{1}{2\xi}}.
      \end{equation}
      For convenience, we denote $\rho_0=|k_j|=\sqrt{-\frac{1}{2\xi}}$;
      \item Region II: For $\xi>0$, $\mathrm{Im}\theta$ and $\mathrm{Im}k$ have the same sign, and there is no stationary point.
\end{enumerate}

In following sections, we discuss the asymptotic behavior for the two regions respectively.

\section{Asymptotic analysis in  the  region  $y/t<0$ } \label{sec3}

In this section,   we investigate the long-time asymptotic behavior in  the  region  $y/t<0$, which includes   two stationary points.
\subsection{Triangular decomposition}
First of all, based on the decay regions as shown in Figure \ref{sign1}, we need to decompose the jump matrix $\widehat{V}(k)$ into upper and lower triangular matrix. We therefore introduce a scalar function $\delta(k)$ satisfying the RH problem as follows.

\begin{rhp}\label{RHd}
Find a scalar function $\delta(k)$ with the following properties:
\begin{itemize}
  \item Analyticity: $\delta(k)$ is analytical in $\mathbb{C}\setminus\mathbb{R};$
  \item Jump condition: $\delta(k)$ has continuous boundary values $\delta_{\pm}$ and
      \begin{equation}
\left\{\begin{array}{ll}
\delta_+(k)=\delta_-(k)(1-|r|^2), &k\in\Gamma_1,\\
\delta_+(k)=\delta_-(k), &k\in(k_1,k_2),\end{array}\right.
\end{equation}
where $\Gamma_1=(-\infty,k_1)\cup(k_2,\infty);$
  \item Asymptotic behavior:
  \begin{equation}
  \delta(k)\to 1,\quad k\to \infty.
  \end{equation}
\end{itemize}
\end{rhp}
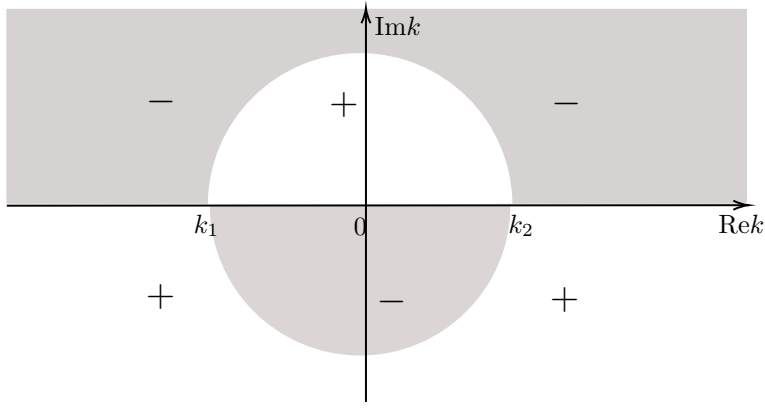
\begin{figure}

\tikzset{every picture/.style={line width=0.75pt}} 

\begin{tikzpicture}[x=0.75pt,y=0.75pt,yscale=-0.7,xscale=0.7]
uncomment if require: \path (-10,100); 

\draw  [color={rgb, 255:red, 255; green, 255; blue, 255 },draw opacity=1 ][fill={rgb, 255:red, 215; green, 210; blue, 210 },fill opacity=1 ] (71.6,0) -- (606.6,0) -- (606.6,142) -- (71.6,142) -- cycle ;
\draw  [color={rgb, 255:red, 255; green, 255; blue, 255 }  ,draw opacity=1 ][fill={rgb, 255:red, 255; green, 255; blue, 255 }  ,fill opacity=1 ] (218.2,142.4) .. controls (218.2,82.2) and (267,33.4) .. (327.2,33.4) .. controls (386.76,33.4) and (435.17,81.18) .. (436.18,140.5) -- (327.2,142.4) -- cycle ;
\draw  [color={rgb, 255:red, 255; green, 255; blue, 255 }  ,draw opacity=1 ][fill={rgb, 255:red, 219; green, 213; blue, 213 }  ,fill opacity=1 ] (436.2,142.4) .. controls (436.2,202.6) and (387.4,251.4) .. (327.2,251.4) .. controls (267,251.4) and (218.2,202.6) .. (218.2,142.4) -- (327.2,142.4) -- cycle ;
\draw    (72.67,142.67) -- (605,142.33) ;
\draw [shift={(607,142.33)}, rotate = 179.96] [color={rgb, 255:red, 0; green, 0; blue, 0 }  ][line width=0.75]    (10.93,-3.29) .. controls (6.95,-1.4) and (3.31,-0.3) .. (0,0) .. controls (3.31,0.3) and (6.95,1.4) .. (10.93,3.29)   ;
\draw    (331.4,284.2) -- (331.4,166.4) -- (331.4,2.2) ;
\draw [shift={(331.4,0.2)}, rotate = 90] [color={rgb, 255:red, 0; green, 0; blue, 0 }  ][line width=0.75]    (10.93,-3.29) .. controls (6.95,-1.4) and (3.31,-0.3) .. (0,0) .. controls (3.31,0.3) and (6.95,1.4) .. (10.93,3.29)   ;

\draw (205.33,146.73) node [anchor=north west][inner sep=0.75pt]    {$k_{1}$};
\draw (432,146.73) node [anchor=north west][inner sep=0.75pt]    {$k_{2}$};
\draw (583.33,147.07) node [anchor=north west][inner sep=0.75pt]    {$\mathrm{Re} k$};
\draw (335.67,4.4) node [anchor=north west][inner sep=0.75pt]    {$\mathrm{Im} k$};
\draw (320.63,149.2) node [anchor=north west][inner sep=0.75pt]    {$0$};
\draw (303,58.4) node [anchor=north west][inner sep=0.75pt]  [font=\Large]  {$+$};
\draw (171,197.4) node [anchor=north west][inner sep=0.75pt]  [font=\Large]  {$+$};
\draw (462,199.4) node [anchor=north west][inner sep=0.75pt]  [font=\Large]  {$+$};
\draw (337,200.4) node [anchor=north west][inner sep=0.75pt]  [font=\Large]  {$-$};
\draw (463,57.4) node [anchor=north west][inner sep=0.75pt]  [font=\Large]  {$-$};
\draw (171,56.4) node [anchor=north west][inner sep=0.75pt]  [font=\Large]  {$-$};

\end{tikzpicture}

  \centering
  \caption{The classification for sign of $\mathrm{Im}\theta$ in the case $\xi<0$. In the shaded regions, $\mathrm{Im}\theta<0$, which implies $e^{- it\theta}\to 0$ as $t\to\infty$; while in the white regions, $\mathrm{Im}\theta>0$ and thus $e^{it\theta}\to 0$ as $t\to\infty.$ }\label{sign1}
\end{figure}
By the Plemelj formula, there exists the unique solution for RH problem \ref{RHd} with
\begin{equation}
\delta(k)=\mathrm{exp}\left[i\int_{\Gamma_1}\frac{\nu(s)}{s-k}\mathrm{d}s\right],
\end{equation}
where
\begin{equation} \nu(s)=-\frac{1}{2\pi}\log(1-|r(s)|^2).
\end{equation}
Therefore, $\delta(k)$ has the following asymptotic expansion at $k=0$:
\begin{equation}
\delta(k)=1+\delta_1(ik)+\frac{\delta_1^2}{2}(ik)^2+\mathcal{O}(k^3),\quad k\to 0,
\end{equation}
where
\begin{equation}
\delta_1=\int_{\Gamma_1}\frac{\nu(s)}{s^2}\mathrm{d}s.
\end{equation}

Define
\begin{equation}
\beta_j(k)=\int_{\Gamma_1}\frac{\nu(s)}{s-k}\mathrm{d}s-\log(k-k_j)\nu(k_j),\ j=1,2,
\end{equation}
then $\delta(k)$ is rewritten as
\begin{equation}
\delta(k)=(k-k_j)^{i\nu(k_j)}e^{i\beta_j(k)}.
\end{equation}
For $r(k)\in H^1(\mathbb{R})$,
\begin{equation}\label{bj}
|\beta_j(k)-\beta_j(k_j)|\lesssim\|r\|_{H^1(\mathbb{R})}|k-k_j|^{1/2}.
\end{equation}
Now we use $\delta(k)$ to define a new matrix function
\begin{equation}\label{trs1}
M^{(1)}(k)\equiv M^{(1)}(y,t,k)=\widehat{M}(k)\delta(k)^{-\sigma_3},
\end{equation}
which is a solution for the following Riemann-Hilbert problem.
\begin{rhp}Find a $2\times 2$ matrix-valued function $M^{(1)}(k)$ with the following properties:
\begin{itemize}
  \item Analyticity: $M^{(1)}(k)$ is analytical in $\mathbb{C}\setminus\mathbb{R};$
  \item Jump condition: $M^{(1)}(k)$ has continuous boundary values $M^{(1)}_{\pm}(k)$ on $\mathbb{R}$ and
\begin{equation}
M^{(1)}_+(k)=M^{(1)}_-(k)V^{(1)}(k),
\end{equation}
where
\begin{equation}
V^{(1)}(k)=\left\{\begin{array}{ll}
\left(\begin{array}{cc}1&0\\ \frac{r}{1-|r|^2}e^{2it\theta}\delta_-^{-2}&1\end{array}\right)
\left(\begin{array}{cc}1&\frac{-\bar{r}}{1-|r|^2}e^{-2it\theta}\delta_+^2\\ 0&1\end{array}\right), &k\in \Gamma_1,\\
\left(\begin{array}{cc}1&-\bar{r}e^{-2it\theta}\delta_-^2\\ 0&1\end{array}\right)\left(\begin{array}{cc}1&0\\
re^{2it\theta}\delta_+^{-2}&1\end{array}\right), &k \in \mathbb{R}\setminus\Gamma_1;
\end{array}\right.
\end{equation}
  \item Asymptotic behaviors:
  $\ M^{(1)}(k)=I+\mathcal{O}(k^{-1}),\ as \ k\to \infty.$
\end{itemize}
\end{rhp}

\subsection{A mixed $\bar{\partial}$-RH problem and its  decomposition}
Through the triangular decomposition, we make a continuous extension to deal with the jump condition.
\begin{figure}
\tikzset{every picture/.style={line width=0.75pt}} 
\begin{tikzpicture}[x=0.75pt,y=0.75pt,yscale=-0.8,xscale=0.8]
uncomment if require: \path (-10,100); 
\tikzset{every picture/.style={line width=0.75pt}} 
\draw  [color={rgb, 255:red, 255; green, 255; blue, 255 }  ,draw opacity=1 ][fill={rgb, 255:red, 215; green, 210; blue, 210 }  ,fill opacity=1 ] (82,26) -- (586,26) -- (586,142) -- (82,142) -- cycle ;
\draw  [color={rgb, 255:red, 255; green, 255; blue, 255 }  ,draw opacity=1 ][fill={rgb, 255:red, 255; green, 255; blue, 255 }  ,fill opacity=1 ] (219.2,141.4) .. controls (219.2,81.2) and (268,32.4) .. (328.2,32.4) .. controls (387.76,32.4) and (436.17,80.18) .. (437.18,139.5) -- (328.2,141.4) -- cycle ;
\draw  [color={rgb, 255:red, 255; green, 255; blue, 255 }  ,draw opacity=1 ][fill={rgb, 255:red, 219; green, 213; blue, 213 }  ,fill opacity=1 ] (437,141.65) .. controls (437,141.65) and (437,141.65) .. (437,141.65) .. controls (437,201.99) and (387.8,250.9) .. (327.1,250.9) .. controls (266.4,250.9) and (217.2,201.99) .. (217.2,141.65) -- (327.1,141.65) -- cycle ;
\draw    (222.67,146) -- (293,216.33) ;
\draw [shift={(262.08,185.41)}, rotate = 225] [color={rgb, 255:red, 0; green, 0; blue, 0 }  ][line width=0.75]    (10.93,-3.29) .. controls (6.95,-1.4) and (3.31,-0.3) .. (0,0) .. controls (3.31,0.3) and (6.95,1.4) .. (10.93,3.29)   ;
\draw    (121.57,44.9) -- (222.67,146) ;
\draw [shift={(176.36,99.69)}, rotate = 225] [color={rgb, 255:red, 0; green, 0; blue, 0 }  ][line width=0.75]    (10.93,-3.29) .. controls (6.95,-1.4) and (3.31,-0.3) .. (0,0) .. controls (3.31,0.3) and (6.95,1.4) .. (10.93,3.29)   ;
\draw    (372,69.9) -- (439,146.33) ;
\draw [shift={(409.46,112.63)}, rotate = 228.76] [color={rgb, 255:red, 0; green, 0; blue, 0 }  ][line width=0.75]    (10.93,-3.29) .. controls (6.95,-1.4) and (3.31,-0.3) .. (0,0) .. controls (3.31,0.3) and (6.95,1.4) .. (10.93,3.29)   ;
\draw    (439,146.33) -- (536.57,243.9) ;
\draw [shift={(492.03,199.36)}, rotate = 225] [color={rgb, 255:red, 0; green, 0; blue, 0 }  ][line width=0.75]    (10.93,-3.29) .. controls (6.95,-1.4) and (3.31,-0.3) .. (0,0) .. controls (3.31,0.3) and (6.95,1.4) .. (10.93,3.29)   ;
\draw    (370,214.9) -- (435.67,143) ;
\draw [shift={(406.88,174.52)}, rotate = 132.41] [color={rgb, 255:red, 0; green, 0; blue, 0 }  ][line width=0.75]    (10.93,-3.29) .. controls (6.95,-1.4) and (3.31,-0.3) .. (0,0) .. controls (3.31,0.3) and (6.95,1.4) .. (10.93,3.29)   ;
\draw    (126,235.33) -- (219,142.33) ;
\draw [shift={(176.74,184.59)}, rotate = 135] [color={rgb, 255:red, 0; green, 0; blue, 0 }  ][line width=0.75]    (10.93,-3.29) .. controls (6.95,-1.4) and (3.31,-0.3) .. (0,0) .. controls (3.31,0.3) and (6.95,1.4) .. (10.93,3.29)   ;
\draw    (435.67,143) -- (538.77,39.9) ;
\draw [shift={(491.46,87.21)}, rotate = 135] [color={rgb, 255:red, 0; green, 0; blue, 0 }  ][line width=0.75]    (10.93,-3.29) .. controls (6.95,-1.4) and (3.31,-0.3) .. (0,0) .. controls (3.31,0.3) and (6.95,1.4) .. (10.93,3.29)   ;
\draw    (72.67,142.67) -- (605,142.33) ;
\draw [shift={(607,142.33)}, rotate = 179.96] [color={rgb, 255:red, 0; green, 0; blue, 0 }  ][line width=0.75]    (10.93,-3.29) .. controls (6.95,-1.4) and (3.31,-0.3) .. (0,0) .. controls (3.31,0.3) and (6.95,1.4) .. (10.93,3.29)   ;
\draw    (331.4,284.2) -- (331.4,166.4) -- (331.4,2.2) ;
\draw [shift={(331.4,0.2)}, rotate = 90] [color={rgb, 255:red, 0; green, 0; blue, 0 }  ][line width=0.75]    (10.93,-3.29) .. controls (6.95,-1.4) and (3.31,-0.3) .. (0,0) .. controls (3.31,0.3) and (6.95,1.4) .. (10.93,3.29)   ;
\draw  [draw opacity=0] (290.37,71.13) .. controls (314.37,86.16) and (330.47,112.65) .. (330.99,142.94) -- (244.55,144.45) -- cycle ; \draw   (290.37,71.13) .. controls (314.37,86.16) and (330.47,112.65) .. (330.99,142.94) ;
\draw    (218,143.33) -- (290,71.33) ;
\draw [shift={(258.24,103.09)}, rotate = 135] [color={rgb, 255:red, 0; green, 0; blue, 0 }  ][line width=0.75]    (10.93,-3.29) .. controls (6.95,-1.4) and (3.31,-0.3) .. (0,0) .. controls (3.31,0.3) and (6.95,1.4) .. (10.93,3.29)   ;
\draw  [draw opacity=0] (330.99,143.94) .. controls (331,144.44) and (331,144.95) .. (331,145.45) .. controls (331,175.29) and (315.89,201.59) .. (292.9,217.13) -- (244.55,145.45) -- cycle ; \draw   (330.99,143.94) .. controls (331,144.44) and (331,144.95) .. (331,145.45) .. controls (331,175.29) and (315.89,201.59) .. (292.9,217.13) ;
\draw  [draw opacity=0] (370.2,214.13) .. controls (347.21,198.59) and (332.1,172.29) .. (332.1,142.45) .. controls (332.1,112.07) and (347.77,85.36) .. (371.46,69.94) -- (418.55,142.45) -- cycle ; \draw   (370.2,214.13) .. controls (347.21,198.59) and (332.1,172.29) .. (332.1,142.45) .. controls (332.1,112.07) and (347.77,85.36) .. (371.46,69.94) ;
\draw [color={rgb, 255:red, 6; green, 6; blue, 6 }  ,draw opacity=1 ]   (317.8,191.15) -- (322.45,182.55) ;
\draw [shift={(323.4,180.79)}, rotate = 118.41] [color={rgb, 255:red, 6; green, 6; blue, 6 }  ,draw opacity=1 ][line width=0.75]    (10.93,-3.29) .. controls (6.95,-1.4) and (3.31,-0.3) .. (0,0) .. controls (3.31,0.3) and (6.95,1.4) .. (10.93,3.29)   ;
\draw [color={rgb, 255:red, 6; green, 6; blue, 6 }  ,draw opacity=1 ]   (344.48,98.23) -- (342.8,100.65) -- (346.69,94.81) ;
\draw [shift={(347.8,93.15)}, rotate = 123.69] [color={rgb, 255:red, 6; green, 6; blue, 6 }  ,draw opacity=1 ][line width=0.75]    (10.93,-3.29) .. controls (6.95,-1.4) and (3.31,-0.3) .. (0,0) .. controls (3.31,0.3) and (6.95,1.4) .. (10.93,3.29)   ;
\draw [color={rgb, 255:red, 6; green, 6; blue, 6 }  ,draw opacity=1 ]   (318.53,100.24) -- (321.51,105.22) ;
\draw [shift={(322.54,106.94)}, rotate = 239.04] [color={rgb, 255:red, 6; green, 6; blue, 6 }  ,draw opacity=1 ][line width=0.75]    (10.93,-3.29) .. controls (6.95,-1.4) and (3.31,-0.3) .. (0,0) .. controls (3.31,0.3) and (6.95,1.4) .. (10.93,3.29)   ;
\draw [color={rgb, 255:red, 6; green, 6; blue, 6 }  ,draw opacity=1 ]   (342.98,184.67) -- (345.97,189.65) ;
\draw [shift={(347,191.37)}, rotate = 239.04] [color={rgb, 255:red, 6; green, 6; blue, 6 }  ,draw opacity=1 ][line width=0.75]    (10.93,-3.29) .. controls (6.95,-1.4) and (3.31,-0.3) .. (0,0) .. controls (3.31,0.3) and (6.95,1.4) .. (10.93,3.29)   ;

\draw (236.53,83.2) node [anchor=north west][inner sep=0.75pt]    {$\Sigma _{11}$};
\draw (180.27,84.2) node [anchor=north west][inner sep=0.75pt]    {$\Sigma _{12}$};
\draw (178.53,177.13) node [anchor=north west][inner sep=0.75pt]    {$\Sigma _{13}$};
\draw (232.6,175.33) node [anchor=north west][inner sep=0.75pt]    {$\Sigma _{14}$};
\draw (452.07,78.87) node [anchor=north west][inner sep=0.75pt]    {$\Sigma _{22}$};
\draw (398.73,79.53) node [anchor=north west][inner sep=0.75pt]    {$\Sigma _{21}$};
\draw (406.2,175.8) node [anchor=north west][inner sep=0.75pt]    {$\Sigma _{24}$};
\draw (450.27,175.27) node [anchor=north west][inner sep=0.75pt]    {$\Sigma _{23}$};
\draw (211.33,147.73) node [anchor=north west][inner sep=0.75pt]    {$k_{1}$};
\draw (430,149.73) node [anchor=north west][inner sep=0.75pt]    {$k_{2}$};
\draw (262.43,118.43) node [anchor=north west][inner sep=0.75pt]    {$\Omega _{11}$};
\draw (165.73,120.53) node [anchor=north west][inner sep=0.75pt]    {$\Omega _{12}$};
\draw (163.73,146.93) node [anchor=north west][inner sep=0.75pt]    {$\Omega _{13}$};
\draw (260.87,146.53) node [anchor=north west][inner sep=0.75pt]    {$\Omega _{14}$};
\draw (463.87,122.13) node [anchor=north west][inner sep=0.75pt]    {$\Omega _{22}$};
\draw (371.27,120.33) node [anchor=north west][inner sep=0.75pt]    {$\Omega _{21}$};
\draw (371.57,145.43) node [anchor=north west][inner sep=0.75pt]    {$\Omega _{24}$};
\draw (464,146.93) node [anchor=north west][inner sep=0.75pt]    {$\Omega _{23}$};
\draw (583.33,147.07) node [anchor=north west][inner sep=0.75pt]    {$\mathrm{Re} k$};
\draw (335.67,4.4) node [anchor=north west][inner sep=0.75pt]    {$\mathrm{Im} k$};
\draw (333.43,63.8) node [anchor=north west][inner sep=0.75pt]    {$\Sigma _{01}$};
\draw (303.93,62.6) node [anchor=north west][inner sep=0.75pt]    {$\Sigma _{02}$};
\draw (305.13,203.8) node [anchor=north west][inner sep=0.75pt]    {$\Sigma _{03}$};
\draw (337.53,202.6) node [anchor=north west][inner sep=0.75pt]    {$\Sigma _{04}$};
\draw (317.63,147.2) node [anchor=north west][inner sep=0.75pt]    {$0$};
\draw (304,29.4) node [anchor=north west][inner sep=0.75pt]  [font=\Large]  {$+$};
\draw (160,205.4) node [anchor=north west][inner sep=0.75pt]  [font=\Large]  {$+$};
\draw (464,207.4) node [anchor=north west][inner sep=0.75pt]  [font=\Large]  {$+$};
\draw (338,218.4) node [anchor=north west][inner sep=0.75pt]  [font=\Large]  {$-$};
\draw (462,40.4) node [anchor=north west][inner sep=0.75pt]  [font=\Large]  {$-$};
\draw (171,45.4) node [anchor=north west][inner sep=0.75pt]  [font=\Large]  {$-$};
\end{tikzpicture}
\caption{The jump lines $\Sigma_{jl}$ and regions $\Omega_{jl}$ for $\xi<-\epsilon$.}
\label{lin1}
\end{figure}
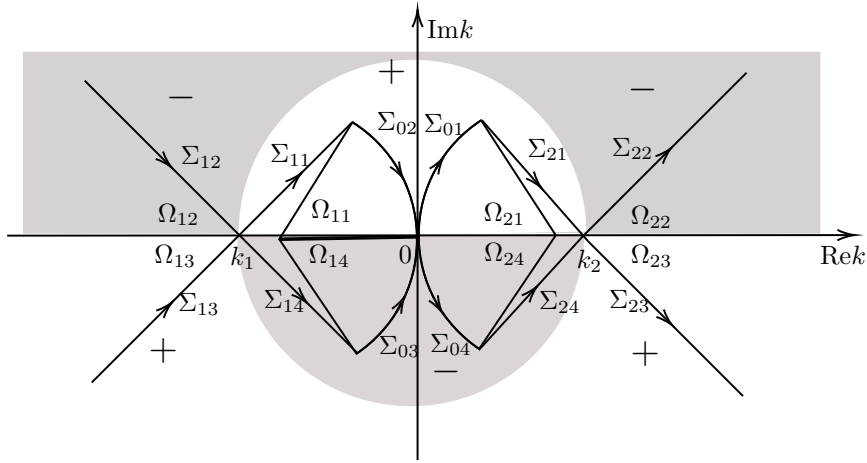
Denote arcs $\Sigma_{0l}$, lines $\Sigma_{jl},$ and domains $\Omega_{jl},\ j=1,2,\ l=1,2,3,4$ as shown in Figure \ref{lin1}. Let
\begin{equation*}
\Sigma_j=\bigcup_{l=1}^4\Sigma_{jl},\ \Sigma=\bigcup_{j=0}^2\Sigma_j,\ \Omega_j=\bigcup_{l=1}^4\Omega_{jl},\ \Omega=\bigcup_{j=1}^2\Omega_j.
\end{equation*}
Now we introduce a new function
\begin{equation}
R^{(2)}(k):=\left\{\begin{array}{ll}
\left(\begin{array}{cc}1&0\\R_{jl}e^{2it\theta}&1\end{array}
\right),&k \in \Omega_{jl},\ j=1,2,\ l=1,3,\\
\left(\begin{array}{cc}1&R_{jl}e^{-2it\theta}\\0&1\end{array}
\right),&k \in \Omega_{jl},\ j=1,2,\ l=2,4,\\
\mathrm{I},&elsewhere,\end{array}\right.
\end{equation}
where $R_{jl}\equiv R_{jl}(k),\ j=1,2,\ l=1,2,3,4$ are functions in $\overline{\Omega}_{jl}$ admitting the properties in the following Proposition:
\begin{proposition}\label{Rjl}
The functions $R_{jl}:\overline{\Omega}_{jl}\to\mathbb{C},\ j=1,2,\ l=1,2,3,4$ have boundary values as follows:
\begin{eqnarray}
&&R_{j1}=\left\{\begin{array}{lr}
-r(k_j)(k-k_j)^{-2i\nu(k_j)}e^{-2i\beta_j(k_j)},\ k\in \Sigma_{j1},\\
-r(k)\delta_+^{-2},\quad k\in \Gamma_{j1},\end{array}\right.\\
&&R_{j2}=\left\{\begin{array}{lr}
\frac{\bar{r}(k_j)}{1-|r(k_j)|^2}(k-k_j)^{2i\nu(k_j)}e^{2i\beta_j(k_j)},\ k\in \Sigma_{j2},\\
\frac{\bar{r}(k)}{1-|r(k)|^2}\delta_+^{2},\quad k\in \Gamma_{j2},\end{array}\right.\\
&&R_{j3}=\left\{\begin{array}{lr}
\frac{r(k_j)}{1-|r(k_j)|^2}(k-k_j)^{ -2i\nu(k_j)}e^{-2i\beta_j(k_j)},\ k\in \Sigma_{j3},\\
\frac{r(k)}{1-|r(k)|^2}\delta_-^{-2},\quad k\in \Gamma_{j2},\end{array}\right.\\
&&R_{j4}=\left\{\begin{array}{lr}
-\bar{r}(k_j)(k-k_j)^{ 2i\nu(k_j)}e^{2i\beta_j(k_j)},\ k\in \Sigma_{j4},\\
-\bar{r}(k)\delta_-^{2},\quad k\in \Gamma_{j1},\end{array}\right.
\end{eqnarray}
where $\Sigma_{jl}$ are shown in Figure \ref{lin1} and $\Gamma_{ij}$ are intervals with
\[
\Gamma_{11}=(k_1,0),\ \Gamma_{21}=(0,k_2), \
\Gamma_{12}=(-\infty,k_1),\ \Gamma_{22}=(k_2,\infty).
\]
For all $j=1,2$, $R_{jl}$ is bounded as $k\in\Omega_{jl}$ and satisfies the estimates:
\begin{eqnarray}
&&|\bar{\partial}R_{jl}|\lesssim \varphi(\mathrm{Re}k)+|k-k_j|^{-1/2}+|r'(\mathrm{Re}k)|,\  for\ all\ k\in\Omega_{jl},\ l=1,4;\quad\qquad \label{pro31}\\
&&|\bar{\partial}R_{jl}|\lesssim |k-k_j|^{-1/2}+|r'(\mathrm{Re}k)|,\quad for\ all\ k\in\Omega_{jl},\ l=2,3;\label{pro32}\\
&&|\bar{\partial}R_{jl}|\lesssim |k|^2,\quad for\ all\ k\in\Omega_{jl},\ l=1,4,\ as\ k\to 0,\label{pro33}
\end{eqnarray}
where $\varphi \in C^{\infty}_0(\mathbb{R},[0,1])$ is a fixed cut-off function with support near 0.
\end{proposition}
\begin{proof}
Since the functions $R_{jl},\ j=1,2,\ l=1,4$ have the same construction, we take $R_{11}$ for example.

Denote $k=k_1+\rho e^{i\alpha}$, for $k\in\Omega_{11}$, $\rho=|k-k_1|\in[0,\rho_0]$, $\alpha\in[0,\pi/4]$. Under the $(\rho,\alpha)$-coordinate, the $\bar{\partial}$-derivative has the following representation
\begin{equation}\label{dbar}
\bar{\partial}=\frac{1}{2}e^{i\alpha}(\partial_{\rho}+i\rho^{-1}\partial_{\alpha}).
\end{equation}
To construct the function $R_{11}$, we introduce a cut-off function $\chi_0(x)\in\ C^{\infty}_0([0,1]),$
\begin{equation}
\chi_0(x)=\left\{\begin{array}{ll}
1,&|x|\leqslant \min\{1,\rho_0\}/8,\\
0,&|x|\geqslant \min\{1,\rho_0\}/4.
\end{array}\right.
\end{equation}
Define the function $R_{11}$ as
\begin{equation*}
R_{11}=R_{11,1}+R_{11,2},
\end{equation*}
where
\begin{equation}
\begin{aligned}
R_{11,1}=&-(1-\chi_0(\mathrm{Re}k))r(\mathrm{Re}k)\delta_+^{-2}cos^2(2\alpha)+\tilde{g}_{11}(1-cos^2(2\alpha)),\\
R_{11,2}=&\tilde{f}(\mathrm{Re}k)cos^2(2\alpha)+\frac{i}{2}\rho e^{-i\alpha}sin(2\alpha)cos(2\alpha)\chi_0(\alpha)\tilde{f}'(\mathrm{Re}k)+sin^2(2\alpha)\chi_0(\alpha)\tilde{f}(\mathrm{Re}k)\\
&-\frac{1}{4}\rho e^{-i\alpha}sin^2(2\alpha)\tilde{f}'(\mathrm{Re}k)-\frac{1}{8}\rho^2 e^{-2i\alpha}sin^2(2\alpha)\tilde{f}''(\mathrm{Re}k),
\end{aligned}
\end{equation}
and
\begin{eqnarray}
&&\tilde{g}_{11}=-r(k_1)(k-k_1)^{-2i\nu(k_1)}e^{-2i\beta_1(k_1)},\nonumber\\
&&\tilde{f}(\mathrm{Re}k)=-\chi_0(\mathrm{Re}k)r(\mathrm{Re}k)\delta_+^{-2}.\nonumber
\end{eqnarray}
Here the function $R_{11,2}$ is built to implement the estimate near $k=0$.

We first deal with $R_{11,1}$. By \eqref{dbar}, we have
\begin{equation}\label{dbar11}
\begin{aligned}
\bar{\partial}R_{11,1}=&\frac{1}{2}\chi_0'(\mathrm{Re}k)r(\mathrm{Re}k)\delta_+^{-2}cos(2\alpha)-\frac{1}{2}[1-\chi_0(\mathrm{Re}k)]r'(\mathrm{Re}k)\delta_+^{-2}cos(2\alpha)\\
&+2i\rho^{-1}e^{i\alpha}[1-\chi_0(\mathrm{Re}k)]r(\mathrm{Re}k)\delta_+^{-2}sin(2\alpha)cos(2\alpha)+2i\rho^{-1}e^{i\alpha}\tilde{g}_{11}sin(2\alpha)cos(2\alpha),
\end{aligned}
\end{equation}
where $\delta^{-2}$, $\rho^{-1},\ r(\mathrm{Re}k)$ and $ r'(\mathrm{Re}k)$ are all bounded in the support of $\chi_0(\mathrm{Re}k)$, thus \eqref{dbar11} is estimated as
\begin{equation}\label{r1-1}
|\bar{\partial}R_{11,1}|\lesssim\varphi(\mathrm{Re}k)+|r'(\mathrm{Re}k)|+\rho^{-1}|\tilde{g}_{11}+r(\mathrm{Re}k)\delta_+^{-2}|.
\end{equation}
The last item of the right is rewritten as
\begin{equation*}
\rho^{-1}|[r(\mathrm{Re}k)-r(k_1)]\delta_+^{-2}+r(k_1)(k-k_1)^{-2i\nu(k_1)}e^{-2i\beta_1(k_1)}[e^{-2i(\beta_1(k)-\beta_1(k_1))}-1]|
\end{equation*}
Based on the estimate $|r(\mathrm{Re}k)-r(k_1)|\lesssim|k-k_1|^{1/2}$ and \eqref{bj}, we finally derive that
\begin{equation}\label{r1-2}
\rho^{-1}|\tilde{g}_{11}+r(\mathrm{Re}k)\delta_+^{-2}|\lesssim|k-k_1|^{-1/2}.
\end{equation}
For $R_{11,2}$, we have
\begin{equation}\label{dbar12}
\begin{aligned}
&\bar{\partial}R_{11,2}=\\
&\left\{\frac{1}{2}\tilde{f}'(\mathrm{Re}k)cos^2(2\alpha)+\frac{i}{4}sin(2\alpha)cos(2\alpha)[8\rho^{-1}e^{i\alpha}\tilde{f}(\mathrm{Re}k)+2\tilde{f}'(\mathrm{Re}k)+\rho e^{-i\alpha}\tilde{f}''(\mathrm{Re}k)]\right\}[\chi_0(\alpha)-1]\\
&+\left\{\frac{i}{2}\rho^{-1}e^{i\alpha}\chi_0'(\alpha)\tilde{f}(\mathrm{Re}k)+[\chi_0(\alpha)-\frac{1}{4}]\tilde{f}'(\mathrm{Re}k)-\frac{3}{8}\rho e^{-i\alpha}\tilde{f}''(\mathrm{Re}k)-\frac{1}{16}\rho^2e^{-2i\alpha}\tilde{f}'''(\mathrm{Re}k)\right\}sin^2(\alpha)\\
&-\frac{1}{4}\rho e^{-i\alpha}sin(2\alpha)cos(2\alpha)\chi_0'(\alpha)\tilde{f}'(\mathrm{Re}k)
\end{aligned}
\end{equation}
Obviously, each item of the right of \eqref{dbar12} is bounded in the support of $\chi_0(\mathrm{Re}k)$, so
\begin{equation}\label{r1-3}
|\bar{\partial}R_{11,2}|\lesssim\varphi(\mathrm{Re}k).
\end{equation}
As $k\to 0$, we have $\alpha\to 0$ and within a small neighborhood of $0$, $\chi_0(\alpha)\equiv1,\ \chi_0'(\alpha)\equiv0$,
thus
\begin{equation}\label{r1-4}
\begin{aligned}
|\bar{\partial}R_{11,2}|&\lesssim|\tilde{f}(\mathrm{Re}k)+\tilde{f}'(\mathrm{Re}k)
+\tilde{f}''(\mathrm{Re}k)+\tilde{f}'''(\mathrm{Re}k)|sin^2(2\alpha)
+|\tilde{f}'(\mathrm{Re}k)\chi_0'(\alpha)sin(2\alpha)|\\
&\lesssim|k|^2,
\end{aligned}
\end{equation}
the last inequality is deduced by Remark \ref{rmk1}, which shows that $r(k),\ r'(k),\ r''(k)$ and $r'''(k)$ are all bounded as $k \to 0$.
Together \eqref{r1-1}, \eqref{r1-2}, \eqref{r1-3} with \eqref{r1-4}, we obtain \eqref{pro31} and \eqref{pro33}.

Meanwhile, $R_{jl},\ j=1,2,\ l=2,3$ are constructed similarly as $R_{11,1}$ without the cut-off function. Taking $R_{12}$ for example:
\begin{equation*}
R_{12}=\frac{\bar{r}(\mathrm{Re}k)}{1-|r(\mathrm{Re}k)|^2}\delta_+^{2}cos(2\alpha)
+\tilde{g}_{12}(1-cos(2\alpha)),
\end{equation*}
where
\begin{equation*}
\tilde{g}_{12}=\frac{\bar{r}(k_1)}{1-|r(k_1)|^2}(k-k_1)^{2i\nu(k_1)}e^{2i\beta_1(k_1)}.
\end{equation*}
Therefore, \eqref{pro32} is proved by the same method used in the estimate of $R_{11,1}$.
\end{proof}

We now use $R^{(2)}(k)$ to define a new unknown function:
\begin{equation}\label{trs2}
M^{(2)}(k)\equiv M^{(2)}(y,t,k)=M^{(1)}(k)R^{(2)}(k).
\end{equation}
Then $M^{(2)}(k)$ satisfies the following mixed $\bar{\partial}$-Riemann-Hilbert problem.
\begin{rhp}\label{RHdbar}
Find a $2\times 2$ matrix-valued function $M^{(2)}(k)$ admitting properties as follows:
\begin{itemize}
  \item Analyticity: $M^{(2)}(k)$ is continuous in $\mathbb{C}\setminus\Sigma;$
  \item Jump condition: $M^{(2)}(k)$ has continuous boundary values $M^{(2)}_{\pm}(k)$ on $\Sigma$ and
  \begin{equation}
  M^{(2)}_+(k)=M^{(2)}_-(k)V^{(2)}(k),
  \end{equation}
  where
  \begin{equation}
  V^{(2)}(k)=\left\{\begin{array}{ll}
  (R^{(2)}(k))^{-1},&k\in\Sigma_{jl},\ j=0,1,2,\ l=1,2,\\
  R^{(2)}(k),&k\in\Sigma_{jl},\ j=0,1,2,\ l=3,4;\end{array}
  \right.
  \end{equation}
  \item Asymptotic behavior:
  \begin{equation}
  M^{(2)}(k)=I+\mathcal{O}(k^{-1}),\ as\ k\to\infty;
  \end{equation}
  \item $\bar{\partial}$-Derivative: For $k\in\mathbb{C}$, we have
  \begin{equation}
  \bar{\partial}M^{(2)}(k)=M^{(2)}(k)\bar{\partial}R^{(2)}(k),
  \end{equation}
  where
  \begin{equation}
 \bar{\partial}R^{(2)}(k)=\left\{
  \begin{array}{ll}
  \left(\begin{array}{cc}
  0&0\\ \bar{\partial}R_{jl}e^{2it\theta}&0\end{array}\right),&k\in\Omega_{jl},\ j=1,2,\ l=1,3,\\
  \left(\begin{array}{cc}
  0&\bar{\partial}R_{jl}e^{-2it\theta}\\ 0&0\end{array}\right),&k\in\Omega_{jl},\ j=1,2,\ l=2,4,\\
  0,&elsewhere.
  \end{array}\right.
  \end{equation}
\end{itemize}
\end{rhp}
Denote
\begin{equation*} U_j=\left\{k\| k-k_j|\leqslant\frac{1}{4}\rho_0\right\},\ j=1,2.
\end{equation*}
and $U=U_1\cup U_2.$ The jump matrix $V^{(2)}(k)$ has the following estimates.
\begin{proposition}\label{propp4}
As $t\to\infty$, for $1\leq p_1\leq +\infty$, there exists a positive constant $c_1(p)$ relied on $p$ such that
\begin{equation}
\| V^{(2)}(k)-I\|_{L^{p}(\Sigma_{jl}\setminus U_j)}=\mathcal{O}(e^{-c_1(p_1)t}),
\end{equation}
for $j=1,2$ and $l=1,2,3,4$. And there exists another positive constant $c_2(p)$ relied on $p$ such that
\begin{equation}\label{l0}
\| V^{(2)}(k)-I\|_{L^{p}(\Sigma_{0l})}=\mathcal{O}(e^{-c_2(p)t}),\ l=1,2,3,4.
\end{equation}
\end{proposition}
\begin{proof}
For $\xi<0$, by \eqref{Imthta}, we find that
\begin{equation*}
\mathrm{Im}\theta=\mathrm{Im}k\frac{2|k|^2\xi+1}{2|k|^2},
\end{equation*}
and $\frac{2|k|^2\xi+1}{2|k|^2}$ has nonzero boundary on $\Sigma_{jl}\setminus U_j$. We take the case as $k\in\Sigma_{11}\setminus U_1$ for example: denote $k=k_1+l e^{\frac{i\pi}{4}},\ l\in(\frac{\rho_0}{4},+\infty)$, then for $1\leq p<+\infty,$
\begin{equation*}
\begin{aligned}
\| V^{(2)}(k)-I\| ^p_{L^p(\Sigma_{11}\setminus U_1)}&=\| -\tilde{g}_{11}e^{2it\theta}\| ^p_{L^p(\Sigma_{11}\setminus U_1)}\lesssim\| e^{2it\theta}\| ^p_{L^p(\Sigma_{11}\setminus U_1)}\\
&\lesssim\int_{\frac{\rho_0}{4}}^{+\infty}e^{-pc'tl}\mathrm{d}l\lesssim t^{-1}e^{-pc't\rho_0/4}.
\end{aligned}
\end{equation*}
For $p=+\infty$, the estimate is obvious.

As for $k\in \Sigma_{01},$ $|V^{(2)}(k)-I|=|R_{11}e^{2it\theta}|$, so \eqref{l0} is deduced by the same way.
\end{proof}
\
To solve RH problem \ref{RHdbar}, we decompose $M^{(2)}(k)$ into a pure RH problem $M^R(k)\equiv M^R(y,t,k)$ with $\bar{\partial}M^R(k)\equiv 0$ and a pure $\bar{\partial}$-problem with nonzero $\bar{\partial}$-derivatives.

We first give a RH problem for the $M^R(k)$:
\begin{rhp}Find a $2\times 2$ matrix-valued function $M^R(k)$ with the following properties:
\begin{itemize}
  \item Analyticity: $M^R(k)$ is analytic in $\mathbb{C}\setminus \Sigma;$
  \item Jump condition: $M^R(k)$ has continuous boundry values $M^R_{\pm}(k)$ on $\Sigma$ and
      \begin{equation}
      M^R_+(k)=M^R_-(k)V^{(2)}(k),\quad k\in\Sigma;
      \end{equation}
  \item Asymptotic behavior:
 $
  M^R(k)=I+\mathcal{O}(k^{-1}),\ as\ k\to\infty;
$
  \item $\bar{\partial}$-Derivative: $\bar{\partial}R^{(2)}=0,$ for all $k\in\mathbb{C}$.
\end{itemize}
\end{rhp}
Next, We use $M^R(k)$ to define a new matrix function
\begin{equation}\label{trs3}
M^{(3)}(k)\equiv M^{(3)}(y,t,k)=M^{(2)}(k)(M^R(k))^{-1},
\end{equation}
which is a pure $\bar{\partial}$-problem by removing $M^R(k)$ from $M^{(2)}(k)$.
\begin{dbar}\label{M3}
 Find a $2\times2$ matrix-valued function $M^{(3)}(k)$ satisfying the following properties:
\begin{itemize}
  \item Continuity: $M^{(3)}(k)$ is continuous in $\mathbb{C};$
  \item Asymptotic behavior:
$
  M^{(3)}(k)\thicksim I+\mathcal{O}(k^{-1}),\quad k\to \infty;
$
  \item $\bar{\partial}$-Derivative: For $k\in \mathbb{C}$, we have
      \begin{equation}
      \bar{\partial}M^{(3)}(k)=M^{(3)}(k)W(k),
      \end{equation}
      where
      \begin{equation}
      W(k)=M^R(k)\bar{\partial}R^{(2)}(k)(M^R(k))^{-1}.
      \end{equation}
\end{itemize}
\end{dbar}

\subsection{Analysis on pure RH problem}\label{3.3}

\subsubsection{A local solvable RH model near stationary points}
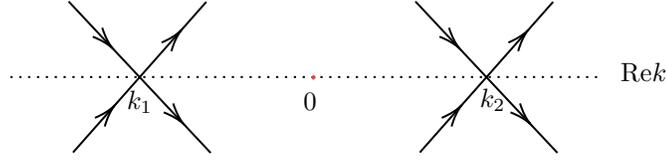
\begin{figure}
\tikzset{every picture/.style={line width=0.75pt}} 
\begin{tikzpicture}[x=0.75pt,y=0.75pt,yscale=-1,xscale=1]
uncomment if require: \path (40,100); 
\draw  [dash pattern={on 0.84pt off 2.51pt}]  (170.67,139) -- (328.02,139) -- (466.67,139) ;
\draw    (235.86,139.08) -- (269.46,101.08) ;
\draw [shift={(256.63,115.59)}, rotate = 131.48] [color={rgb, 255:red, 0; green, 0; blue, 0 }  ][line width=0.75]    (10.93,-3.29) .. controls (6.95,-1.4) and (3.31,-0.3) .. (0,0) .. controls (3.31,0.3) and (6.95,1.4) .. (10.93,3.29)   ;
\draw    (200.06,100.88) -- (235.86,139.08) ;
\draw [shift={(222.06,124.36)}, rotate = 226.86] [color={rgb, 255:red, 0; green, 0; blue, 0 }  ][line width=0.75]    (10.93,-3.29) .. controls (6.95,-1.4) and (3.31,-0.3) .. (0,0) .. controls (3.31,0.3) and (6.95,1.4) .. (10.93,3.29)   ;
\draw    (235.86,139.08) -- (271.66,177.28) ;
\draw [shift={(257.86,162.56)}, rotate = 226.86] [color={rgb, 255:red, 0; green, 0; blue, 0 }  ][line width=0.75]    (10.93,-3.29) .. controls (6.95,-1.4) and (3.31,-0.3) .. (0,0) .. controls (3.31,0.3) and (6.95,1.4) .. (10.93,3.29)   ;
\draw    (202.26,177.08) -- (235.86,139.08) ;
\draw [shift={(223.03,153.59)}, rotate = 131.48] [color={rgb, 255:red, 0; green, 0; blue, 0 }  ][line width=0.75]    (10.93,-3.29) .. controls (6.95,-1.4) and (3.31,-0.3) .. (0,0) .. controls (3.31,0.3) and (6.95,1.4) .. (10.93,3.29)   ;
\draw    (410.76,139.08) -- (444.36,101.08) ;
\draw [shift={(431.53,115.59)}, rotate = 131.48] [color={rgb, 255:red, 0; green, 0; blue, 0 }  ][line width=0.75]    (10.93,-3.29) .. controls (6.95,-1.4) and (3.31,-0.3) .. (0,0) .. controls (3.31,0.3) and (6.95,1.4) .. (10.93,3.29)   ;
\draw    (374.96,100.88) -- (410.76,139.08) ;
\draw [shift={(396.96,124.36)}, rotate = 226.86] [color={rgb, 255:red, 0; green, 0; blue, 0 }  ][line width=0.75]    (10.93,-3.29) .. controls (6.95,-1.4) and (3.31,-0.3) .. (0,0) .. controls (3.31,0.3) and (6.95,1.4) .. (10.93,3.29)   ;
\draw    (410.76,139.08) -- (446.56,177.28) ;
\draw [shift={(432.76,162.56)}, rotate = 226.86] [color={rgb, 255:red, 0; green, 0; blue, 0 }  ][line width=0.75]    (10.93,-3.29) .. controls (6.95,-1.4) and (3.31,-0.3) .. (0,0) .. controls (3.31,0.3) and (6.95,1.4) .. (10.93,3.29)   ;
\draw    (377.16,177.08) -- (410.76,139.08) ;
\draw [shift={(397.93,153.59)}, rotate = 131.48] [color={rgb, 255:red, 0; green, 0; blue, 0 }  ][line width=0.75]    (10.93,-3.29) .. controls (6.95,-1.4) and (3.31,-0.3) .. (0,0) .. controls (3.31,0.3) and (6.95,1.4) .. (10.93,3.29)   ;
\draw    (100,123) ;
\draw    (100,101) ;
\draw [color={rgb, 255:red, 247; green, 57; blue, 57 }  ,draw opacity=1 ][line width=1.5]    (322.96,138.63) -- (324.14,139.81) ;

\draw (228,144.4) node [anchor=north west][inner sep=0.75pt]    {$k_{1}$};
\draw (406,143.4) node [anchor=north west][inner sep=0.75pt]    {$k_{2}$};
\draw (317,145.4) node [anchor=north west][inner sep=0.75pt]    {$0$};
\draw (477,130.4) node [anchor=north west][inner sep=0.75pt]    {$\mathrm{Re} k$};
\end{tikzpicture}
\caption{The jump lines $\Sigma^{lo}$.}
\label{lo}
\end{figure}
To solve the RH problem $M^R(k)$, we use the PC model to estimate the behavior near stationary points, which will be defined as a new local RH problem $M^{lo}(k)\equiv M^{lo}(y,t,k)$. And then we find the error function between $M^R(k)$ and $M^{lo}(k)$ to deduce the asymptotic representation of $M^R(k)$.

To introduce the local function, we denote several lines as shown in Figure \ref{lo}
\begin{equation*}
\Sigma_{jl}^{lo}=\Sigma
_{jl}\cap U_j,\ j=1,2,\ l=1,2,3,4,
\end{equation*}
and
\begin{equation*}
\Sigma^{lo}=\Sigma^{lo}_1\cup\Sigma^{lo}_2,\quad \Sigma^{lo}_j=\cup_{l=1}^4\Sigma_{jl}^{lo}.
\end{equation*}
Now we consider the following RH problem
\begin{rhp}\label{rhplo}
Find a $2\times 2$ matrix-valued function $M^{lo}(k)$ admitting the properties as follows:
\begin{itemize}
  \item Analyticity: $M^{lo}(k)$ is analytical in $\mathbb{C}\setminus\Sigma^{lo};$
  \item Jump condition: $M^{lo}(k)$ has continuous boundary values $M^{lo}_{\pm}(k)$ on $\Sigma^{lo}$ and
      \begin{equation}
      M^{lo}_+(k)=M^{lo}_-(k)V^{(2)}(k),\quad k\in \Sigma^{lo};
      \end{equation}
  \item Asymptotic behavior:
  $
  M^{lo}(k)=I+\mathcal{O}(k^{-1}),\quad k\to\infty.
$
\end{itemize}
\end{rhp}
For $j=1,2$ and $l=1,2,3,4$, denote
\begin{equation*}
\omega_{jl}(k)=\left\{\begin{array}{ll}
\left(\begin{array}{cc}
  0&0\\ (-1)^{\frac{l-1}{2}}R_{jl}e^{2it\theta}&0\end{array}\right),&k\in\Sigma_{jl},\ j=1,2,\ l=1,3,\\
  \left(\begin{array}{cc}
  0&(-1)^{\frac{l-2}{2}}R_{jl}e^{-2it\theta}\\ 0&0\end{array}\right),&k\in\Sigma_{jl},\ j=1,2,\ l=2,4,
\end{array}\right.
\end{equation*}
then $V^{(2)}(k)=I-\omega_{jl}(k)$ for $k\in\Sigma_{jl}$. Let
\begin{equation*}
\omega_j(k)=\sum_{l=1}^4\omega_{jl}(k),\ \omega(k)=\omega_1+\omega_2,
\end{equation*}
and
\begin{equation*}
\omega_{jl}^{\pm}(k)=\omega_{jl}(k)|_{\mathbb{C}^{\pm}},\ \omega_{j}^{\pm}(k)=\omega_{j}(k)|_{\mathbb{C}^{\pm}},\ \omega^{\pm}(k)=\omega(k)|_{\mathbb{C}^{\pm}}.
\end{equation*}
Considering the Cauchy projection operator $\mathcal{C}_{\pm}$ on $\Sigma$,
we define the following operator:
\begin{equation}\label{cw}
\mathcal{C}_{\omega}(f)=\mathcal{C}_+(f\omega^-)+\mathcal{C}_-(f\omega^+),\ \mathcal{C}_{\omega_j}(f)=\mathcal{C}_+(f\omega_j^-)+\mathcal{C}_-(f\omega_j^+).
\end{equation}
Thus $\mathcal{C}_{\omega}=\mathcal{C}_{\omega_1}+\mathcal{C}_{\omega_2}$. \begin{lemma}\label{wjl}
For the function $\omega_{jl}(k),\ j=1,2,\ l=1,2,3,4$, we have following estimate:
\begin{equation}
\| \omega_{jl}(k)\|_{L^2}\lesssim t^{-1/2},
\end{equation}
\end{lemma}
This Lemma can be directly proved by calculation. And thus $I-\mathcal{C}_{\omega}$ and $I-\mathcal{C}_{\omega_j}$ are reversible. Therefore, there exists unique solution of RH problem \ref{rhplo}, and it is denoted as
\begin{equation}
M^{lo}(k)=I+\frac{1}{2\pi i}\int_{\Sigma^{lo}}\frac{(I-\mathcal{C}_{\omega})^{-1}I\omega}{s-k}\mathrm{d}s.
\end{equation}
Besides, by \eqref{cw} and Lemma \ref{wjl}, we deduce following corollary:
\begin{corollary}As $t\to\infty$,
\begin{equation*}
\| \mathcal{C}_{\omega_j}\mathcal{C}_{\omega_l}\|_{B(L^2(\Sigma^{lo}))}\lesssim t^{-1},\quad \| \mathcal{C}_{\omega_j}\mathcal{C}_{\omega_l}\|_{(L^{\infty}(\Sigma^{lo}\to L^2(\Sigma^{lo})))}\lesssim t^{-1}.
\end{equation*}
\end{corollary}
Therefore, by above results, we finally find that the contributions of every crosses $\Sigma^{lo}_j,\ j=1,2$ are separated out.

The RH problems near the two stationary points are estimated by PC model respectively. Here we take the estimate near $k_1$ for example. Consider the following RH problem:
\begin{rhp}Find a $2\times 2$ matrix-valued function $M^{lo,1}(k)\equiv M^{lo,1}(y,t,k)$ satisfying the following properties:
\begin{itemize}
  \item Analyticity: $M^{lo,1}(k)$ is analytical in $\mathbb{C}\setminus\Sigma^{lo}_1;$
  \item Jump condition: $M^{lo,1}(k)$ has continuous boundary values $M^{lo,1}_{\pm}(k)$ on $\Sigma^{lo}_1$, and
      \begin{equation}
      M^{lo,1}_+(k)=M^{lo,1}_-(k)V^{lo,1}(k),\quad k\in \Sigma_1^{lo},
      \end{equation}
  where
  \begin{equation}
  V^{lo,1}(k)=\left\{\begin{array}{ll}
  \left(\begin{array}{cc}1&0\\r(k_1)(k-k_1)^{-2i\nu(k_1)}e^{-2i\beta_1(k_1)}e^{2it\theta}&1\end{array}\right),&k\in\Sigma_{11}^{lo},\\
  \left(\begin{array}{cc}1&-\frac{\overline{r(k_1)}}{1-|r(k_1)|^2}(k-k_1)^{2i\nu(k_1)}e^{2i\beta_1(k_1)}e^{-2it\theta}\\0&1\end{array}\right),&k\in\Sigma_{12}^{lo},\\
  \left(\begin{array}{cc}1&0\\\frac{r(k_1)}{1-|r(k_1)|^2}(k-k_1)^{-2i\nu(k_1)}e^{-2i\beta_1(k_1)}e^{2it\theta}&1\end{array}\right),&k\in\Sigma_{13}^{lo},\\
  \left(\begin{array}{cc}1&-\overline{r(k_1)}(k-k_1)^{2i\nu(k_1)}e^{2i\beta_1(k_1)}e^{-2it\theta}\\0&1\end{array}\right),&k\in\Sigma_{14}^{lo};
  \end{array}\right.
  \end{equation}
  \item Asymptotic behavior:
$
  M^{lo,1}(k)=I+\mathcal{O}(k^{-1}),\quad k\to\infty.
$
\end{itemize}
\end{rhp}
Taking Cauchy expansion at $k=k_1$ for $\theta$, we get
\begin{equation}
\theta=-\frac{1}{k_1}-\frac{1}{k_1^3}(k-k_1)^2+\mathcal{O}((k-k_1)^3).
\end{equation}
Let $\zeta=\zeta(k)$ defined as follows:
\begin{equation}
\zeta=t^{1/2}\sqrt{-\frac{4}{k_1^3}}(k-k_1),
\end{equation}
and denote a parameter $\widetilde{r}_1$ as
\begin{equation}
\widetilde{r}_1=r(k_1)e^{-2i(\beta_1(k_1)+tk_1^{-1})}\mathrm{exp}\{i\nu(k_1)\mathrm{log}(-4tk_1^{-3})\}.
\end{equation}
Now $M^{lo,1}$ can match the following model RH problem:
\begin{rhp}\label{PC}
Find a $2\times 2$ matrix-valued function $M^{pc}(\zeta)$ with the following properties:
\begin{itemize}
  \item Asymptotic: $M^{pc}(\zeta)$ is analytical in $\mathbb{C}\setminus\Sigma^{pc}$ with $\Sigma^{pc}={\mathbb{R}e^{\varphi i}}\cup{\mathbb{R}e^{(\pi-\varphi) i}};$
  \item Jump condition: $M^{pc}(\zeta)$ has continuous boundary values $M_{\pm}^{pc}$ on $\Sigma^{pc}$ and
      \begin{equation}
      M^{pc}_+(\zeta)=M^{pc}_-(\zeta)V^{pc}(\zeta),\quad \zeta\in\Sigma^{pc},
      \end{equation}
      where
      \begin{equation}
V^{pc}(\zeta)=\left\{\begin{array}{ll}
\left(\begin{array}{cc}
1&0\\ \widetilde{r}_1\zeta^{-2i\nu(k_1)}e^{\frac{i}{2}\zeta^2}&1
\end{array}\right),&\zeta\in \mathbb{R}^+e^{\varphi i},\\
\left(\begin{array}{cc}
1&-\frac{\overline{\widetilde{r}_1}}{1-|\widetilde{r}_1|^2}\zeta^{2i\nu(k_1)}e^{-\frac{i}{2}\zeta^2}\\0 &1
\end{array}\right),&\zeta\in \mathbb{R}^+e^{(\pi-\varphi) i},\\
\left(\begin{array}{cc}
1&0\\ \frac{\widetilde{r}_1}{1-|\widetilde{r}_1|^2}\zeta^{-2i\nu(k_1)}e^{\frac{i}{2}\zeta^2}&1
\end{array}\right),&\zeta\in \mathbb{R}^+e^{(\pi+\varphi) i},\\
\left(\begin{array}{cc}
1&-\overline{\widetilde{r}_1}\zeta^{2i\nu(k_1)}e^{-\frac{i}{2}\zeta^2}\\ 0&1
\end{array}\right),&\zeta\in \mathbb{R}^+e^{-\varphi i};\\
\end{array}\right.
\end{equation}
  \item Asymptotic behavior:
  \begin{equation}\label{pca}
M^{pc}(\zeta)=I+\frac{M_1^{pc}}{\zeta}+\mathcal{O}(\zeta^{-2}),\quad \zeta\to\infty.
\end{equation}
\end{itemize}
\end{rhp}
There exists the unique solution of RH problem \ref{PC} with the asymptotic behavior as \eqref{pca}, where
\begin{equation*}
M_1^{pc}=\left(\begin{array}{cc}0&-i\beta^1_{12}\\i\beta^1_{21}&0\end{array}\right),
\end{equation*}
and
\begin{equation}
\beta^1_{12}=\frac{\sqrt{2\pi}e^{-\frac{\pi}{2}\nu(k_1)}e^{\frac{\pi}{4}i}}{\widetilde{r}_1\Gamma(-i\nu(k_1))},\quad \beta^1_{21}=\frac{\nu(k_1)}{\beta^1_{12}}.
\end{equation}

As proved in A1-A6 of \cite{HG}, $M^{lo,1}(k)$ is estimated as:
\begin{equation*}
M^{lo,1}(k)=I+t^{-\frac{1}{2}}\frac{\rho_0^{3/2}}{k-k_1
}\frac{i}{2}\left(\begin{array}{cc}0&-\beta^1_{12}\\ \beta^1_{21}&0\end{array}\right)+\mathcal{O}(t^{-1}).
\end{equation*}
Similarly,
\begin{equation*}
M^{lo,2}(k)=I+t^{-\frac{1}{2}}\frac{\rho_0^{3/2}}{k-k_2}\frac{i}{2}\left(\begin{array}{cc}0&-\beta^2_{12}\\ \beta^2_{21}&0\end{array}\right)+\mathcal{O}(t^{-1}),
\end{equation*}
where
\begin{equation}
\beta^2_{12}=\frac{\sqrt{2\pi}e^{\frac{3\pi}{2}\nu(k_2)}e^{\frac{3\pi}{4}i}}{\widetilde{r}_2\Gamma(i\nu(k_2))},\quad \beta^2_{21}=\frac{\nu(k_2)}{\beta^2_{12}},
\end{equation}
and
\begin{equation}
\widetilde{r}_2=r(k_2)e^{-2i(\beta_2(k_2)+tk_2^{-1})}\mathrm{exp}\{i\nu(k_2)\mathrm{log}(-4tk_2^{-3})\}.
\end{equation}
Therefore, we have
\begin{equation}\label{mlo}
M^{lo}(k)=I+t^{\frac{1}{2}}\frac{\rho_0^{3/2}i}{2}\sum_{j=1}^2\frac{B_j}{k-k_j}+\mathcal{O}(t^{-1}),
\end{equation}
where
\begin{equation}
B_j=\left(\begin{array}{cc}0&-\beta_{12}^j\\ \beta_{21}^j&0\end{array}\right),\quad j=1,2.
\end{equation}

\subsubsection{Error function with small norm }
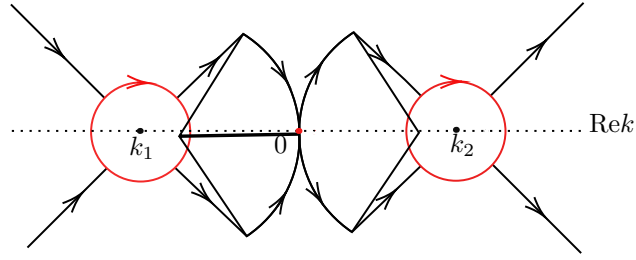
\begin{figure}
\tikzset{every picture/.style={line width=0.75pt}} 

\begin{tikzpicture}[x=0.75pt,y=0.75pt,yscale=-1,xscale=1]
uncomment if require: \path (40,100); 
\tikzset{every picture/.style={line width=0.75pt}} 

\draw    (100,123) ;
\draw    (100,101) ;
\draw [color={rgb, 255:red, 247; green, 57; blue, 57 }  ,draw opacity=1 ][line width=1.5]    (322.96,138.63) -- (324.14,139.81) ;
\draw    (170.44,74.7) -- (215.9,120.16) -- (245.09,149.35) -- (289.36,191.67) ;
\draw [shift={(197.41,101.67)}, rotate = 225] [color={rgb, 255:red, 0; green, 0; blue, 0 }  ][line width=0.75]    (10.93,-3.29) .. controls (6.95,-1.4) and (3.31,-0.3) .. (0,0) .. controls (3.31,0.3) and (6.95,1.4) .. (10.93,3.29)   ;
\draw [shift={(234.74,139)}, rotate = 225] [color={rgb, 255:red, 0; green, 0; blue, 0 }  ][line width=0.75]    (10.93,-3.29) .. controls (6.95,-1.4) and (3.31,-0.3) .. (0,0) .. controls (3.31,0.3) and (6.95,1.4) .. (10.93,3.29)   ;
\draw [shift={(271.56,174.66)}, rotate = 223.71] [color={rgb, 255:red, 0; green, 0; blue, 0 }  ][line width=0.75]    (10.93,-3.29) .. controls (6.95,-1.4) and (3.31,-0.3) .. (0,0) .. controls (3.31,0.3) and (6.95,1.4) .. (10.93,3.29)   ;
\draw    (178.77,197.4) -- (225.82,150.35) -- (254.82,121.35) -- (287.59,89.58) ;
\draw [shift={(206.53,169.63)}, rotate = 135] [color={rgb, 255:red, 0; green, 0; blue, 0 }  ][line width=0.75]    (10.93,-3.29) .. controls (6.95,-1.4) and (3.31,-0.3) .. (0,0) .. controls (3.31,0.3) and (6.95,1.4) .. (10.93,3.29)   ;
\draw [shift={(244.56,131.61)}, rotate = 135] [color={rgb, 255:red, 0; green, 0; blue, 0 }  ][line width=0.75]    (10.93,-3.29) .. controls (6.95,-1.4) and (3.31,-0.3) .. (0,0) .. controls (3.31,0.3) and (6.95,1.4) .. (10.93,3.29)   ;
\draw [shift={(275.51,101.29)}, rotate = 135.89] [color={rgb, 255:red, 0; green, 0; blue, 0 }  ][line width=0.75]    (10.93,-3.29) .. controls (6.95,-1.4) and (3.31,-0.3) .. (0,0) .. controls (3.31,0.3) and (6.95,1.4) .. (10.93,3.29)   ;
\draw    (343.29,88.75) -- (369.26,112.85) -- (395.71,139.31) -- (413.16,155.35) -- (458.01,200.2) ;
\draw [shift={(360.67,104.88)}, rotate = 222.88] [color={rgb, 255:red, 0; green, 0; blue, 0 }  ][line width=0.75]    (10.93,-3.29) .. controls (6.95,-1.4) and (3.31,-0.3) .. (0,0) .. controls (3.31,0.3) and (6.95,1.4) .. (10.93,3.29)   ;
\draw [shift={(386.73,130.32)}, rotate = 225] [color={rgb, 255:red, 0; green, 0; blue, 0 }  ][line width=0.75]    (10.93,-3.29) .. controls (6.95,-1.4) and (3.31,-0.3) .. (0,0) .. controls (3.31,0.3) and (6.95,1.4) .. (10.93,3.29)   ;
\draw [shift={(408.85,151.39)}, rotate = 222.59] [color={rgb, 255:red, 0; green, 0; blue, 0 }  ][line width=0.75]    (10.93,-3.29) .. controls (6.95,-1.4) and (3.31,-0.3) .. (0,0) .. controls (3.31,0.3) and (6.95,1.4) .. (10.93,3.29)   ;
\draw [shift={(439.83,182.02)}, rotate = 225] [color={rgb, 255:red, 0; green, 0; blue, 0 }  ][line width=0.75]    (10.93,-3.29) .. controls (6.95,-1.4) and (3.31,-0.3) .. (0,0) .. controls (3.31,0.3) and (6.95,1.4) .. (10.93,3.29)   ;
\draw    (342.42,189.57) -- (378.09,155.35) -- (413.34,119.78) -- (459.24,73.88) ;
\draw [shift={(364.58,168.31)}, rotate = 136.19] [color={rgb, 255:red, 0; green, 0; blue, 0 }  ][line width=0.75]    (10.93,-3.29) .. controls (6.95,-1.4) and (3.31,-0.3) .. (0,0) .. controls (3.31,0.3) and (6.95,1.4) .. (10.93,3.29)   ;
\draw [shift={(399.94,133.31)}, rotate = 134.75] [color={rgb, 255:red, 0; green, 0; blue, 0 }  ][line width=0.75]    (10.93,-3.29) .. controls (6.95,-1.4) and (3.31,-0.3) .. (0,0) .. controls (3.31,0.3) and (6.95,1.4) .. (10.93,3.29)   ;
\draw [shift={(440.53,92.59)}, rotate = 135] [color={rgb, 255:red, 0; green, 0; blue, 0 }  ][line width=0.75]    (10.93,-3.29) .. controls (6.95,-1.4) and (3.31,-0.3) .. (0,0) .. controls (3.31,0.3) and (6.95,1.4) .. (10.93,3.29)   ;
\draw  [color={rgb, 255:red, 243; green, 41; blue, 41 }  ,draw opacity=1 ][fill={rgb, 255:red, 255; green, 255; blue, 255 }  ,fill opacity=1 ] (210.86,139.08) .. controls (210.86,125.27) and (222.05,114.08) .. (235.86,114.08) .. controls (249.67,114.08) and (260.86,125.27) .. (260.86,139.08) .. controls (260.86,152.89) and (249.67,164.08) .. (235.86,164.08) .. controls (222.05,164.08) and (210.86,152.89) .. (210.86,139.08) -- cycle ;
\draw  [color={rgb, 255:red, 235; green, 32; blue, 32 }  ,draw opacity=1 ][fill={rgb, 255:red, 255; green, 255; blue, 255 }  ,fill opacity=1 ] (369.86,138.68) .. controls (369.86,124.87) and (381.05,113.68) .. (394.86,113.68) .. controls (408.67,113.68) and (419.86,124.87) .. (419.86,138.68) .. controls (419.86,152.49) and (408.67,163.68) .. (394.86,163.68) .. controls (381.05,163.68) and (369.86,152.49) .. (369.86,138.68) -- cycle ;
\draw  [dash pattern={on 0.84pt off 2.51pt}]  (170.88,138.53) -- (323.94,138.53) -- (458.8,138.53) ;
\draw  [fill={rgb, 255:red, 9; green, 8; blue, 8 }  ,fill opacity=1 ] (235.35,138.62) .. controls (235.35,138.36) and (235.56,138.15) .. (235.82,138.15) .. controls (236.07,138.15) and (236.28,138.36) .. (236.28,138.62) .. controls (236.28,138.87) and (236.07,139.08) .. (235.82,139.08) .. controls (235.56,139.08) and (235.35,138.87) .. (235.35,138.62) -- cycle ;
\draw  [fill={rgb, 255:red, 9; green, 8; blue, 8 }  ,fill opacity=1 ] (409.86,138.68) .. controls (409.86,138.43) and (410.07,138.22) .. (410.32,138.22) .. controls (410.58,138.22) and (410.78,138.43) .. (410.78,138.68) .. controls (410.78,138.93) and (410.58,139.14) .. (410.32,139.14) .. controls (410.07,139.14) and (409.86,138.93) .. (409.86,138.68) -- cycle ;
\draw [color={rgb, 255:red, 243; green, 15; blue, 15 }  ,draw opacity=1 ]   (233.86,114.08) -- (237.86,114.28) ;
\draw [shift={(239.86,114.38)}, rotate = 182.86] [color={rgb, 255:red, 243; green, 15; blue, 15 }  ,draw opacity=1 ][line width=0.75]    (10.93,-3.29) .. controls (6.95,-1.4) and (3.31,-0.3) .. (0,0) .. controls (3.31,0.3) and (6.95,1.4) .. (10.93,3.29)   ;
\draw [color={rgb, 255:red, 243; green, 15; blue, 15 }  ,draw opacity=1 ]   (392.86,113.68) -- (395.96,113.66) ;
\draw [shift={(397.96,113.64)}, rotate = 179.55] [color={rgb, 255:red, 243; green, 15; blue, 15 }  ,draw opacity=1 ][line width=0.75]    (10.93,-3.29) .. controls (6.95,-1.4) and (3.31,-0.3) .. (0,0) .. controls (3.31,0.3) and (6.95,1.4) .. (10.93,3.29)   ;
\draw  [draw opacity=0] (287.59,89.58) .. controls (304.38,100.09) and (315.63,118.61) .. (315.99,139.79) -- (255.55,140.85) -- cycle ; \draw   (287.59,89.58) .. controls (304.38,100.09) and (315.63,118.61) .. (315.99,139.79) ;
\draw  [draw opacity=0] (315.99,140.49) .. controls (316,140.84) and (316,141.2) .. (316,141.55) .. controls (316,162.41) and (305.43,180.81) .. (289.36,191.67) -- (255.55,141.55) -- cycle ; \draw   (315.99,140.49) .. controls (316,140.84) and (316,141.2) .. (316,141.55) .. controls (316,162.41) and (305.43,180.81) .. (289.36,191.67) ;
\draw  [draw opacity=0] (342.42,189.57) .. controls (326.34,178.71) and (315.77,160.31) .. (315.77,139.45) .. controls (315.77,118.21) and (326.73,99.53) .. (343.29,88.75) -- (376.22,139.45) -- cycle ; \draw   (342.42,189.57) .. controls (326.34,178.71) and (315.77,160.31) .. (315.77,139.45) .. controls (315.77,118.21) and (326.73,99.53) .. (343.29,88.75) ;
\draw  [color={rgb, 255:red, 255; green, 255; blue, 255 }  ,draw opacity=1 ][line width=0.75] [line join = round][line cap = round] (323.35,139.88) .. controls (323.35,139.88) and (323.35,139.88) .. (323.35,139.88) ;
\draw  [color={rgb, 255:red, 255; green, 255; blue, 255 }  ,draw opacity=1 ][line width=0.75] [line join = round][line cap = round] (322.85,139.87) .. controls (323.18,139.79) and (323.52,139.71) .. (323.85,139.63) ;
\draw  [color={rgb, 255:red, 255; green, 255; blue, 255 }  ,draw opacity=1 ][line width=0.75] [line join = round][line cap = round] (325.1,138.63) .. controls (322.86,140.55) and (322.76,140.62) .. (324.85,138.88) ;
\draw  [color={rgb, 255:red, 255; green, 255; blue, 255 }  ,draw opacity=1 ][line width=0.75] [line join = round][line cap = round] (322.1,138.62) .. controls (322.43,138.96) and (322.77,139.29) .. (323.1,139.63) ;
\draw  [color={rgb, 255:red, 255; green, 255; blue, 255 }  ,draw opacity=1 ][line width=0.75] [line join = round][line cap = round] (324.1,139.12) .. controls (324.1,139.29) and (324.1,139.46) .. (324.1,139.63) ;
\draw  [color={rgb, 255:red, 255; green, 255; blue, 255 }  ,draw opacity=1 ][line width=0.75] [line join = round][line cap = round] (323.1,139.38) .. controls (323.35,139.46) and (323.6,139.54) .. (323.85,139.62) ;
\draw  [color={rgb, 255:red, 255; green, 255; blue, 255 }  ,draw opacity=1 ][line width=0.75] [line join = round][line cap = round] (323.1,139.38) .. controls (323.35,139.46) and (323.6,139.54) .. (323.85,139.62) ;
\draw  [color={rgb, 255:red, 255; green, 255; blue, 255 }  ,draw opacity=1 ][line width=0.75] [line join = round][line cap = round] (410.4,138.47) .. controls (410.4,140.95) and (410.4,139.45) .. (410.4,136.97) ;
\draw  [color={rgb, 255:red, 255; green, 255; blue, 255 }  ,draw opacity=1 ][line width=0.75] [line join = round][line cap = round] (410.4,138.23) .. controls (410.82,139.06) and (410.82,138.81) .. (409.65,138.23) ;
\draw  [color={rgb, 255:red, 255; green, 255; blue, 255 }  ,draw opacity=1 ][line width=0.75] [line join = round][line cap = round] (409.65,139.22) .. controls (410.07,138.72) and (410.48,138.22) .. (410.9,137.73) ;
\draw  [color={rgb, 255:red, 255; green, 255; blue, 255 }  ,draw opacity=1 ][line width=0.75] [line join = round][line cap = round] (410.65,139.47) .. controls (410.65,138.81) and (410.65,138.14) .. (410.65,137.48) ;
\draw  [color={rgb, 255:red, 255; green, 255; blue, 255 }  ,draw opacity=1 ][line width=0.75] [line join = round][line cap = round] (410.9,139.47) .. controls (410.82,138.89) and (410.73,138.31) .. (410.65,137.72) ;
\draw  [color={rgb, 255:red, 255; green, 255; blue, 255 }  ,draw opacity=1 ][line width=0.75] [line join = round][line cap = round] (411.15,140.73) .. controls (410.9,139.31) and (410.65,137.89) .. (410.4,136.47) ;
\draw  [color={rgb, 255:red, 242; green, 25; blue, 25 }  ,draw opacity=1 ][line width=2.25] [line join = round][line cap = round] (315.4,138.54) .. controls (315.59,138.54) and (315.78,138.54) .. (315.97,138.54) ;
\draw  [color={rgb, 255:red, 12; green, 12; blue, 12 }  ,draw opacity=1 ][line width=2.25] [line join = round][line cap = round] (394.83,137.97) .. controls (394.83,137.97) and (394.83,137.97) .. (394.83,137.97) ;
\draw  [color={rgb, 255:red, 12; green, 12; blue, 12 }  ,draw opacity=1 ][line width=2.25] [line join = round][line cap = round] (235.4,138.54) .. controls (235.4,138.54) and (235.4,138.54) .. (235.4,138.54) ;
\draw [color={rgb, 255:red, 6; green, 6; blue, 6 }  ,draw opacity=1 ]   (304.84,176.6) -- (308.02,171.39) ;
\draw [shift={(309.06,169.68)}, rotate = 121.38] [color={rgb, 255:red, 6; green, 6; blue, 6 }  ,draw opacity=1 ][line width=0.75]    (10.93,-3.29) .. controls (6.95,-1.4) and (3.31,-0.3) .. (0,0) .. controls (3.31,0.3) and (6.95,1.4) .. (10.93,3.29)   ;
\draw [color={rgb, 255:red, 6; green, 6; blue, 6 }  ,draw opacity=1 ]   (325.24,107.8) -- (323.64,109.8) -- (326.16,105.88) ;
\draw [shift={(327.24,104.2)}, rotate = 122.74] [color={rgb, 255:red, 6; green, 6; blue, 6 }  ,draw opacity=1 ][line width=0.75]    (10.93,-3.29) .. controls (6.95,-1.4) and (3.31,-0.3) .. (0,0) .. controls (3.31,0.3) and (6.95,1.4) .. (10.93,3.29)   ;
\draw [color={rgb, 255:red, 6; green, 6; blue, 6 }  ,draw opacity=1 ]   (307.64,110.2) -- (309.01,112.49) ;
\draw [shift={(310.04,114.2)}, rotate = 239.04] [color={rgb, 255:red, 6; green, 6; blue, 6 }  ,draw opacity=1 ][line width=0.75]    (10.93,-3.29) .. controls (6.95,-1.4) and (3.31,-0.3) .. (0,0) .. controls (3.31,0.3) and (6.95,1.4) .. (10.93,3.29)   ;
\draw [color={rgb, 255:red, 6; green, 6; blue, 6 }  ,draw opacity=1 ]   (325.24,171.4) -- (326.61,173.69) ;
\draw [shift={(327.64,175.4)}, rotate = 239.04] [color={rgb, 255:red, 6; green, 6; blue, 6 }  ,draw opacity=1 ][line width=0.75]    (10.93,-3.29) .. controls (6.95,-1.4) and (3.31,-0.3) .. (0,0) .. controls (3.31,0.3) and (6.95,1.4) .. (10.93,3.29)   ;

\draw (228.49,140.16) node [anchor=north west][inner sep=0.75pt]    {$k_{1}$};
\draw (390.33,138.4) node [anchor=north west][inner sep=0.75pt]    {$k_{2}$};
\draw (301.75,139.4) node [anchor=north west][inner sep=0.75pt]    {$0$};
\draw (460.5,128.4) node [anchor=north west][inner sep=0.75pt]    {$\mathrm{Re} k$};
\end{tikzpicture}
\caption{The jump lines $\Sigma^{(E)}$ for $E(k)$.}
\label{Elin}
\end{figure}

Considering the error between $M^R(k)$ and $M^{lo}(k)$, we define a new function as follows:
\begin{equation}\label{trs4}
 E(k)\equiv E(y,t,k)=\left\{\begin{array}{ll}
M^R(k),&k\in \mathbb{C}\setminus U,\\
M^R(k)(M^{lo}(k))^{-1},&k\in U.
\end{array}\right.
\end{equation}
Then $E(k)$ is a solution of the new RH problem:
\begin{rhp}\label{rhpe}
Find a $2\times 2$ matrix-valued function $E(k)$ with following properties:
\begin{itemize}
  \item Analyticity: $E(k)$ is analytical in $\mathbb{C}\setminus\Sigma^{(E)}$, where
      $
      \Sigma^{(E)}=\partial U  \cup (\Sigma\setminus U)
     $
      as shown in Figure \ref{Elin}$;$
  \item Jump condition: $E(k)$ has continuous boundary values $E_{\pm}(k)$ on $\Sigma^{(E)}$, and
      \begin{equation}
      E_+(k)=E_-(k)V^{(E)}(k),
      \end{equation}
      where
      \begin{equation}\label{ve}
      V^{(E)}(k)=\left\{\begin{array}{ll}
      V^{(2)},&k\in\Sigma\setminus U,\\
      M^{lo}(k),&k\in \partial U;
      \end{array}\right.
      \end{equation}
  \item Asymptotic behavior:
$
  E(k)\sim I+\mathcal{O}(k^{-1}),\quad k\to\infty.
$
\end{itemize}
\end{rhp}
Using Proposition \ref{propp4}, we get the estimates of $V^{(E)}(k)$ as follows:
\begin{equation}
\| V^{(E)}(k)-I\|_p\lesssim\left\{\begin{array}{ll}
e^{-c_1(p)t},&k\in \Sigma_{jl}\setminus U,\ j=1,2,\\
e^{-c_2(p)t},&k\in \Sigma_{0l},\ l=1,2,3,4.
\end{array}\right.
\end{equation}
For $k\in\partial U$, by \eqref{mlo} and \eqref{ve}, we have
\begin{equation}
|V^{(E)}(k)-I|=|M^{lo}(k)-I|=\mathcal{O}(t^{-1/2}).
\end{equation}
Therefore, based on the Beals-Coifman theorem, by the same method in Section 3, the solution of the RH problem \ref{rhpe} is represented by
\begin{equation}
E(k)=I+\frac{1}{2\pi i}\int_{\Sigma_E}\frac{\mu_E(s)(V^{(E)}(s)-I)}{s-k}\mathrm{d}s,
\end{equation}
where $\mu_E\in L^2(\Sigma^{(E)})$ satisfies
\begin{equation*}
(1-\mathcal{C}_{\omega_E})\mu_E=I,
\end{equation*}
and the operator $\mathcal{C}_{\omega_E}$ is defined by
\begin{equation*} \mathcal{C}_{\omega_E}f=\mathcal{C}_-(f(V^{(E)}-I))
\end{equation*}
with the estimate:
\begin{equation*}
\| \mathcal{C}_{\omega_E}\|_{L^2(\Sigma^{(E)})}\leqslant\| \mathcal{C}_-\|_{L^2(\Sigma^{(E)})}
\| V^{(E)}-I\|_{L^{\infty}(\Sigma^{(E)})}\leqslant\mathcal{O}(t^{-1/2}),
\end{equation*}
which guarantees the existence and uniqueness of $\mu_E$ and $M^{(E)}$.
We finally give the asymptotic expansion of $E(k)$ at $k=0$:
\begin{proposition}\label{estiE}
As $k\to 0$, we have
\begin{equation}
E(k)=E_0+E_1k+E_2k^2+\mathcal{O}(k^3),
\end{equation}
where
\begin{eqnarray}
&&E_0=E(0)=I+\frac{1}{2\pi i}\int_{\Sigma^{(E)}}\frac{\mu_E(s)(V^{(E)}(s)-I)}{s}\mathrm{d}s,\nonumber\\
&&E_1=-\frac{1}{2\pi i}\int_{\Sigma^{(E)}}\frac{\mu_E(s)(V^{(E)}(s)-I)}{s^2}\mathrm{d}s,\quad E_2=\frac{1}{2\pi i}\int_{\Sigma^{(E)}}\frac{\mu_E(s)(V^{(E)}(s)-I)}{s^3}\mathrm{d}s.\nonumber
\end{eqnarray}
Besides, $E_j,\ j=0,1,2$ satisfies following long-time asymptotic behavior:
\begin{eqnarray}
&&E_0-I=t^{-\frac{1}{2}}H_0+\mathcal{O}(t^{-1}),\label{E0}\\
&&E_j=t^{-\frac{1}{2}}H_j+\mathcal{O}(t^{-1}),\label{E12}\ j=1,2,
\end{eqnarray}
and by the
residue theorem, we write them out as:
\begin{equation}
H_j=(-1)^j\frac{\rho_0^{3/2}i}{2}\sum_{l=1}^2\frac{B_l}{k_l^{j+1}}.
\end{equation}
\end{proposition}
\begin{proof}
The proof of \eqref{E0} and \eqref{E12} are similar, so we only give the details for the first one.
Notice that
\begin{equation*}
\frac{1}{s-k}=\frac{1}{s}\sum_{n=0}^{\infty}(\frac{k}{s})^n,
\end{equation*}
thus
\begin{equation*}
\begin{aligned}
&E_0-I=\frac{1}{2\pi i}\int_{\Sigma^{(E)}}\frac{\mu_E(V^{(E)}-I)}{s}\mathrm{d}s\\
=&\frac{1}{2\pi i}\int_{\Sigma^{(E)}}\frac{(\mu_E-I)(V^{(E)}-I)}{s}\mathrm{d}s+\frac{1}{2\pi i}\int_{\Sigma^{(E)}\setminus U}\frac{V^{(E)}-I}{s}\mathrm{d}s+\frac{1}{2\pi i}\int_{\partial U}\frac{V^{(E)}-I}{s}\mathrm{d}s\\
=&\frac{1}{2\pi i}\sum_{j=1}^2\int_{\partial U_j}t^{-\frac{1}{2}}\frac{i\rho_0^{\frac{3}{2}}B_j}{2(s-k_j)}s^{-1}\mathrm{d}s+\mathcal{O}(t^{-1})=\frac{i\rho_0^{\frac{3}{2}}}{2}\sum_{j=1}^2\frac{B_j}{k_j}t^{-\frac{1}{2}}+\mathcal{O}(t^{-1})\\
=&H_0t^{-\frac{1}{2}}+\mathcal{O}(t^{-1}).
\end{aligned}
\end{equation*}
\end{proof}

\subsection{Analysis on the pure $\bar{\partial}$-problem}\label{3.4}

The solution $M^{(3)}(k)$ of the $\bar{\partial}$-problem \ref{M3} is given by the integral equation:
\begin{equation}\label{M3eq}
M^{(3)}(k)=I+\frac{1}{\pi}\iint_{\mathbb{C}}\frac{M^{(3)}(s)W^{(3)}(s)}{s-k}\mathrm{d}A(s).
\end{equation}
Denote the integral operator $S^{(3)}$ as
\begin{equation*}
S^{(3)}[f](k)=\frac{1}{\pi}\iint_{\mathbb{C}}
\frac{f(s)W^{(3)}(s)}{s-k}\mathrm{d}A(s).
\end{equation*}
Then, \eqref{M3eq} equals to
\begin{equation}
(I-S^{(3)})M^{(3)}(k)=I.
\end{equation}
To prove the existence of $(I-S^{(3)})^{-1}$, we need the following proposition.
\begin{proposition}\label{s-op}
As $t\to\infty$, the norm of the integral operator $S^{(3)}$ satisfies the estimate:
\begin{equation}
\| S^{(3)}\|_{L^{\infty}\to L^{\infty}}\lesssim t^{-1/2}.
\end{equation}
\end{proposition}
\begin{proof}
We only give the proof of $k\in \Omega_{11}$. Define a new function $f(k)\in L^{\infty}(\Omega_{11})$, then
\begin{equation*}
\begin{aligned}
|S(f)|&\leqslant\frac{1}{\pi}\iint_{\Omega_{11}}\frac{|f(s)M^R(s)\bar{\partial}R^{(2)}(s)(M^R(s))^{-1}
      |}{|s-k|}\mathrm{d}A(s)\\
      &\leqslant\| f\|_{L^{\infty}}\frac{1}{\pi}\iint_{\Omega_{11}}\frac{|\bar{\partial}R_{12}e^{-2it\theta}|}{|s-k|}\mathrm{d}A(s)\lesssim I_1+I_2+I_3,
\end{aligned}
\end{equation*}
where
\begin{eqnarray}
&&I_1=\iint_{\Omega_{11}}\frac{|\varphi(\mathrm{Re}s)e^{-2it\theta}|}{|s-k|}\mathrm{d}A(s),\ I_2=\iint_{\Omega_{11}}\frac{|r'(\mathrm{Re}s)e^{-2it\theta}|}{|s-k|}\mathrm{d}A(s),\nonumber\\
&&I_3=\iint_{\Omega_{11}}\frac{|s-k_1|^{-\frac{1}{2}}|e^{-2it\theta}|}{|s-k|}\mathrm{d}A(s).\nonumber
\end{eqnarray}
Denote $s=k_1+u+iv,\ k=\alpha+i\eta$, by Lemma \ref{Imk} and $\varphi\in C_0^{\infty}$, we have
\begin{equation*}
I_1\leqslant\int_0^{\infty}\int_{v}^{\infty}\frac{|\varphi(u)|}{|s-k|}e^{2\xi tv}\mathrm{d}u\mathrm{d}v\leqslant\int_0^{\infty}e^{-ctv}\| \varphi\|_{L^2}\| (s-k)^{-1}\|_{L^2}\mathrm{d}v,
\end{equation*}
where
\begin{equation}\label{s-k}
\begin{aligned}
\| (s-k)^{-1}\|_{L^2(v,\infty)}&\leqslant\left(\int_{-\infty}^{\infty}\frac{1}{(k_1+u-\alpha)^2+(v-\eta)^2}\mathrm{d}u\right)^{1/2}\\
&=\left(\frac{1}{|v-\eta|}\int_{-\infty}^{\infty}\frac{1}{1+y^2}\mathrm{d}y\right)^{1/2}=\left(\frac{\pi}{|v-\eta|}\right)^{1/2}
\end{aligned}
\end{equation}
with $y=\frac{k_1+u-\alpha}{v-\eta}$. Thus
\begin{equation*}
I_1\lesssim\int_0^{\infty}\frac{e^{-ctv}}{\sqrt{|v-\eta|}}\mathrm{d}v\lesssim t^{-1/2}.
\end{equation*}
The estimate of $I_2$ is just the same as $I_1$ because $r'(k)\in L^2(\mathbb{R})$, so we get
$
I_2\lesssim t^{-1/2}.
$

Finally, we deal with $I_3$. We first give the proof of the estimates as follows: for $2< p<\infty$ and $q$ satisfying $\frac{1}{p}+\frac{1}{q}=1$, we have
\begin{equation}\label{s-k1}
\| \frac{1}{\sqrt{|s-k_1|}}\|_{L^p(v,\infty )}=\left(\int_{v}^{\infty}\frac{1}{|u+iv|^{p/2}}\mathrm{d}u\right)^{1/p}
=v^{\frac{1}{p}-\frac{1}{2}}\left(\int_{1}^{\infty}\frac{1}{(1+y^2)^{p/4}}\mathrm{d}y\right)^{1/p}\lesssim v^{\frac{1}{p}-\frac{1}{2}},
\end{equation}
and by the same method as \eqref{s-k}, we find
\begin{equation}\label{s-k2}
\| (s-k)^{-1}\|_{L^q(v,\infty)}\lesssim|v-\eta|^{1/q-1}.
\end{equation}
So
\begin{equation}
I_3\leqslant\int_0^{+\infty}e^{-ctv}\int^{\infty}_{v}\frac{|s-1|^{-1/2}}{|s-k|}\mathrm{d}u\mathrm{d}v
\lesssim\int_0^{+\infty}e^{-ctv}v^{\frac{1}{p}-\frac{1}{2}}|v-\eta|^{\frac{1}{q}-1}\mathrm{d}v\lesssim t^{-1/2}.
\end{equation}
\end{proof}
Therefore, the solution of $\bar{\partial}$-problem \ref{M3} is unique and obeys the formula \eqref{M3eq}. Then we have the following asymptotic estimation of $M^{(3)}(k)$.
\begin{proposition}\label{estiM3}
As $k \rightarrow 0$, $M^{(3)}(k)$ has the asymptotic expansion:
\begin{equation}
M^{(3)}(k)=M^{(3)}(0)+M^{(3)}_1k+M^{(3)}_2k^2+O(k^3),
\end{equation}
where
\begin{equation}\label{m30}
|M^{(3)}(0)-I|=|\frac{1}{\pi}\iint_{\mathbb{C}}
\frac{M^{(3)}(s)W^{(3)}(s)}{s}\mathrm{d}A(s)|\lesssim t^{-1},
\end{equation}
and
\begin{equation}
M^{(3)}_1=-\frac{1}{\pi}\iint_{\mathbb{C}}
\frac{M^{(2)}(s)W^{(3)}(s)}{s^2}\mathrm{d}A(s),\quad
M^{(3)}_2=-\frac{1}{\pi}\iint_{\mathbb{C}}
\frac{M^{(2)}(s)W^{(3)}(s)}{s^3}\mathrm{d}A(s)
\end{equation}
satisfying
\begin{equation}\label{m31}
|M^{(3)}_1|\lesssim t^{-1},\ as\ t\rightarrow \infty,
\end{equation}
and for $2<p<\infty$,
\begin{equation}\label{M32}
|M^{(3)}_2|\lesssim t^{-1+\frac{1}{2p}},\ as\ t\rightarrow \infty.
\end{equation}
\end{proposition}
\begin{proof}
We take the case as $s\in \Omega_{11}$ for example. Considering the estimate \eqref{pro33} in Proposition \ref{Rjl}, we divide the integral region into $\Omega_{11}\cap B(0)$, $\Omega_{11}\cap B_{k_1}$ and $\Omega'_{11}:=\Omega_{11}\setminus (B(0)\cap B_{k_1})$, where $B(0)=\{k\mid|k|<\epsilon'_0<\rho_0/4\}$ satisfying
\begin{equation*}
|\bar{\partial}R_{11}|\lesssim|k|^2,\ for\ all\ k\in B(0),
\end{equation*}
and $B_{k_1}=\{k\mid |k-k_1|<\rho_0/4\}$.
Then
\begin{equation}\label{M30e}
|M^{(3)}(0)-I|_{\Omega_{11}}\lesssim(\iint_{\Omega_{11}\cap B(0)}+\iint_{\Omega_{11}\cap B_{k_1}}+\iint_{\Omega'_{11}})\frac{|\bar{\partial}R_{11}|e^{-ct\mathrm{Im}\theta}}{|s|}\mathrm{d}A(s),
\end{equation}
where the first part
\begin{equation}
\iint_{\Omega_{11}\cap B(0)}\frac{|\bar{\partial}R_{11}|e^{-ctv}}{|s|}\mathrm{d}A(s)\lesssim\iint_{\Omega_2\cap B(0)}e^{-ctv}\mathrm{d}A(s)\lesssim t^{-1}.
\end{equation}
For the second part, $|s|>\rho_0/2$, so we have
\begin{equation*}
\begin{aligned}
&\iint_{\Omega_{11}\cap B_{k_1}}\frac{|\bar{\partial}R_{11}|e^{-ctv}}{|s|}\mathrm{d}A(s)\\
\lesssim&\iint_{\Omega_{11}\cap B_{k_1}}(|\varphi(u)|+|r'(u)|)e^{-ctv}\mathrm{d}A(s)+\iint_{\Omega_{11}\cap B_{k_1}}|s-k_1|^{-1/2}e^{ct\mathrm{Im}\theta}\mathrm{d}A(s)
:=I^{(0)}_1+I^{(0)}_2.
\end{aligned}
\end{equation*}
Notice that $\varphi,\ r'(k)\in L^1(\mathbb{R})$, thus
\begin{equation}
I^{(0)}_1\lesssim\iint_{\Omega_{11}}(|\varphi(u)|+|r'(u)|)e^{-ctv}\mathrm{d}A(s)\lesssim\int_0^{+\infty}(\| \varphi\|_{L^1}+\| r'\|_{L^1})e^{-ctv}\mathrm{d}v\lesssim t^{-1}.
\end{equation}
As for $I^{(0)}_2$, based on the definition of the region $B_{k_1}$ and \eqref{Imthta}, we deduce that
\begin{equation}\label{boundedim}
\mathrm{Im}\theta=uv\left(\frac{\xi}{u}+\frac{1}{2u[(u+k_1)^2+v^2]}\right)\leqslant
c_{\xi}uv
\end{equation}
where $c_{\xi}$ is a constant independent with $u,v$. Therefore, by \eqref{s-k1}, we have
\begin{equation*}
\begin{aligned}
I^{(0)}_2&\lesssim \int_0^{+\infty}\| |s-k_1|^{-1/2}\|_{L^q}\| e^{-c_{\xi}tuv}\|_{L^p}\mathrm{d}v\lesssim t^{-\frac{1}{p}}\int_0^{+\infty}v^{\frac{2}{q}-\frac{1}{2}}e^{-ctv^2}\mathrm{d}v\\
&=t^{-\frac{1}{p}}(\int_0^1+\int_1^{+\infty})v^{\frac{2}{q}-\frac{1}{2}}e^{-ctv^2}\mathrm{d}v=t^{-\frac{1}{p}}(I^{(0)}_{21}+I^{(0)}_{22})
\end{aligned}
\end{equation*}
with
\begin{equation*}
I^{(0)}_{21}\lesssim t^{-\frac{1}{4}-\frac{1}{q}}\int_0^{\infty}y^{\frac{1}{q}-\frac{3}{4}}e^{-y}\mathrm{d}y\lesssim t^{-\frac{1}{4}-\frac{1}{q}},
\end{equation*}
and
\begin{equation*}
I^{(0)}_{22}\lesssim \int_1^{+\infty}v^{\frac{2}{q}-\frac{1}{2}}e^{-ctv}\mathrm{d}v\lesssim t^{-\frac{1}{2}-\frac{2}{q}}.
\end{equation*}
Therefore,
\begin{equation}
I^{(0)}_2\lesssim t^{-\frac{5}{4}}.
\end{equation}
For the last part of \eqref{M30e}, we have
\begin{equation}
\begin{aligned}
\iint_{\Omega'_{11}}\frac{|\bar{\partial}R_{11}|e^{-ct\mathrm{Im}\theta}}{|s|}\mathrm{d}A(s)&\lesssim
\int_0^{\infty}\int_v^{\infty}(|\varphi(u)|+|r'(u)|)e^{-ctv}\mathrm{d}u\mathrm{d}v+\int_{\frac{\rho_0}{4}}^{\infty}\int_v^{\infty}\frac{|s-k_1|^{-\frac{1}{2}}}{|s|}e^{-ctv}\mathrm{d}u\mathrm{d}v\\
&\lesssim \int_0^{\infty}e^{-ctv}(\| \varphi(u)\|_{L^1}+\| r'(u)\|_{L^1})\mathrm{d}v+\int_{\frac{\rho_0}{4}}^{\infty}e^{-ctv}v^{-1/2}\mathrm{d}v\lesssim t^{-1}.
\end{aligned}
\end{equation}
Summing up all above, we proved \eqref{m30}. The estimate \eqref{m31} is deduced similarly.

Finally, we give the proof of \eqref{M32}:
\begin{equation}
|M^{(3)}_2|\leqslant \iint_{\Omega_{11}\cap B(0)}\frac{e^{-ct\mathrm{Im}\theta}}{|s|}\mathrm{d}A(s)+(\iint_{\Omega_{11}\cap B_{k_1}}+\iint_{\Omega'_{11}})\frac{|\bar{\partial}R_{11}|e^{-ct\mathrm{Im}\theta}}{|s|}\mathrm{d}A(s),
\end{equation}
where the last two parts have been estimated to decay by the speed of at least $t^{-1}$ as $t\to\infty$. Now we only need to deal with the first item:

As for $|s|<\rho_0/4,$ there holds $|s-k_1|\leqslant \rho_0/4+|k_1|$. For $s\in \Omega_{11}\cap B(0)$, \eqref{boundedim} holds for some constant $c'_{\xi}$. Based on above, taking $p>2,\ k=0$ in \eqref{s-k2}, we find
\begin{equation}
\iint_{\Omega_{11}\cap B(0)}\frac{e^{-ct\mathrm{Im}\theta}}{|s|}\mathrm{d}A(s)\lesssim
\int_0^{\rho_0/4}\| s^{-1}\|_{L^p}\| e^{-c'_{\xi}tuv}\|_{L^q}\mathrm{d}v\lesssim t^{-\frac{1}{q}}\int_0^{\rho_0/4}|v|^{\frac{1}{p}-1}e^{-c'_{\xi}tv^2}\mathrm{d}v
\lesssim t^{-1+\frac{1}{2p}}.
\end{equation}
Thus, we conclude the proof.
\end{proof}

\subsection{Long-time asymptotics  for the HS equation  }
Finally, we construct the long-time asymptotic approximation for the solution   of the HS equation \eqref{mhs}. Inverting the transformations \eqref{trs1}, \eqref{trs2}, \eqref{trs3} and \eqref{trs4}, we have
\begin{equation}
\widehat{M}(k)=M^{(3)}(k)E(k)(R^{(2)})^{-1}(k)\delta^{\sigma_3}(k).
\end{equation}
We take $k\to 0$ out of $\Omega$ so that $R^{(2)}(k)=I$. Then by the results of Proposition \ref{estiE} and Proposition \ref{estiM3}, we obtain the asymptotic expansion of $\widehat{M}(k)$ as follows:
\begin{equation}
\begin{aligned}
\widehat{M}(k)=&M^{(3)}(k)E(k)\delta^{\sigma_3}=(I+\mathcal{O}(t^{-1+\frac{1}{2p}}))(E_0+E_1k+E_2k^2)(I+i\delta_1\sigma_3k-\frac{\delta_1^2}{2}I k^2)+\mathcal{O}(k^3)\\
=&(I+t^{-\frac{1}{2}}(H_0+H_1k+H_2k^2))
(I+i\delta_1\sigma_3k-\frac{\delta_1^2}{2}I k^2)+\mathcal{O}(t^{-1+\frac{1}{2p}})+\mathcal{O}(k^3),
\end{aligned}
\end{equation}
which implies that
\begin{eqnarray}
&&\widehat{M}_{1,1}(y,t)=i\delta_1-\frac{t^{-\frac{1}{2}}}{2}\rho_0^{3/2}f_{11}+\mathcal{O}(t^{-1}),\ \qquad
\widehat{M}_{2,1}(y,t)=-i\delta_1-\frac{t^{-\frac{1}{2}}}{2}\rho_0^{3/2}f_{12}+\mathcal{O}(t^{-1}),\qquad\qquad\nonumber\\
&&\widehat{M}_{1,2}(y,t)=-\frac{1}{2}\delta_1^2-\frac{t^{-\frac{1}{2}}}{2}\rho_0^{3/2}f_{21}+\mathcal{O}(t^{-1}),\quad
\widehat{M}_{2,2}(y,t)=-\frac{1}{2}\delta_1^2-\frac{t^{-\frac{1}{2}}}{2}\rho_0^{3/2}f_{22}+\mathcal{O}(t^{-1}),\nonumber
\end{eqnarray}
where
\begin{eqnarray}
&&f_{11}=\delta_1\sum_{j=1}^2\frac{\beta_{21}^j}{k_j}+i\sum_{j=1}^2\frac{\beta_{21}^j}{k_j^2},\label{f11}\\ &&f_{12}=\delta_1\sum_{j=1}^2\frac{\beta_{12}^j}{k_j}-i\sum_{j=1}^2\frac{\beta_{12}^j}{k_j^2},\\
&&f_{21}=\frac{i}{2}\delta_1^2\sum_{j=1}^2\frac{\beta_{21}^j}{k_j}-\delta_1\sum_{j=1}^2\frac{\beta_{21}^j}{k_j^2}-i\sum_{j=1}^2\frac{\beta_{21}^j}{k_j^3},\\
&&f_{21}=\frac{i}{2}\delta_1^2\sum_{j=1}^2\frac{\beta_{12}^j}{k_j}-\delta_1\sum_{j=1}^2\frac{\beta_{12}^j}{k_j^2}+i\sum_{j=1}^2\frac{\beta_{12}^j}{k_j^3}.
\end{eqnarray}
Therefore, by the reconstruction formula \eqref{rstr-u}, we conclude that
\begin{equation}
\begin{aligned}
u(x,t)&=\hat{u}(y(x,t),t)
=\widehat{M}_{1,1}\widehat{M}_{2,1}(y(x,t),t)+\widehat{M}_{1,2}(y(x,t),t)+\widehat{M}_{2,2}(y(x,t),t)\\
&=\hat{f}t^{-1/2}+\mathcal{O}(t^{-1+\frac{1}{2p}}),
\end{aligned}
\end{equation}
where
\begin{equation}\label{hatf}
\hat{f}=-\frac{\rho_0^{\frac{3}{2}}}{2}\left(i\delta_1(f_{12}-f_{11})+f_{21}+f_{22}\right),
\end{equation}
and
\begin{equation}
x(y,t)=y-\frac{1}{2i}
\widehat{M}_{1,1}(y,t)=y-\frac{1}{2}\delta_1+\frac{\rho_0^{3/2}}{2}f_{11}t^{-\frac{1}{2}}+\mathcal{O}(t^{-1+\frac{1}{2p}}).
\end{equation}

\section{Asymptotic analysis in  the  region  $y/t>0$ }\label{sec4}

\begin{figure}
\tikzset{every picture/.style={line width=0.75pt}} 
\begin{tikzpicture}[x=0.75pt,y=0.75pt,yscale=-1,xscale=1]
uncomment if require: \path (50,100); 
\draw    (234,54) -- (430,236) ;
\draw    (250,236) -- (418,55) ;
\draw    (149,146.33) -- (508,146.33) ;
\draw [shift={(510,146.33)}, rotate = 180] [color={rgb, 255:red, 0; green, 0; blue, 0 }  ][line width=0.75]    (10.93,-3.29) .. controls (6.95,-1.4) and (3.31,-0.3) .. (0,0) .. controls (3.31,0.3) and (6.95,1.4) .. (10.93,3.29)   ;
\draw    (280,254) -- (280,31) ;
\draw [shift={(280,29)}, rotate = 90] [color={rgb, 255:red, 0; green, 0; blue, 0 }  ][line width=0.75]    (10.93,-3.29) .. controls (6.95,-1.4) and (3.31,-0.3) .. (0,0) .. controls (3.31,0.3) and (6.95,1.4) .. (10.93,3.29)   ;

\draw (514,137.4) node [anchor=north west][inner sep=0.75pt]    {$\mathrm{Re} k$};
\draw (267,7.6) node [anchor=north west][inner sep=0.75pt]    {$\mathrm{Im} k$};
\draw (267,147.4) node [anchor=north west][inner sep=0.75pt]    {$0$};
\draw (327.9,147.73) node [anchor=north west][inner sep=0.75pt]    {$1$};
\draw (421,44.4) node [anchor=north west][inner sep=0.75pt]    {$\Sigma_{1}$};
\draw (211,42.4) node [anchor=north west][inner sep=0.75pt]    {$\Sigma_{2}$};
\draw (225,230.4) node [anchor=north west][inner sep=0.75pt]    {$\Sigma_{3}$};
\draw (434,229.4) node [anchor=north west][inner sep=0.75pt]    {$\Sigma_{4}$};
\draw (384,117.4) node [anchor=north west][inner sep=0.75pt]    {$\Omega_{1}$};
\draw (386,163.4) node [anchor=north west][inner sep=0.75pt]    {$\Omega_{4}$};
\draw (243,167.4) node [anchor=north west][inner sep=0.75pt]    {$\Omega_{3}$};
\draw (241,116.4) node [anchor=north west][inner sep=0.75pt]    {$\Omega_{2}$};
\draw (166,94.4) node [anchor=north west][inner sep=0.75pt]  [font=\large,color={rgb, 255:red, 243; green, 17; blue, 17 }  ,opacity=1 ]  {$+$};
\draw (167,186.4) node [anchor=north west][inner sep=0.75pt]  [font=\large,color={rgb, 255:red, 236; green, 19; blue, 19 }  ,opacity=1 ]  {$-$};
\end{tikzpicture}
\caption{The signature of $\mathrm{Im}\theta$ in the case $\xi>0$ and the jump lines $\Sigma_j$}
\label{line2}
\end{figure}
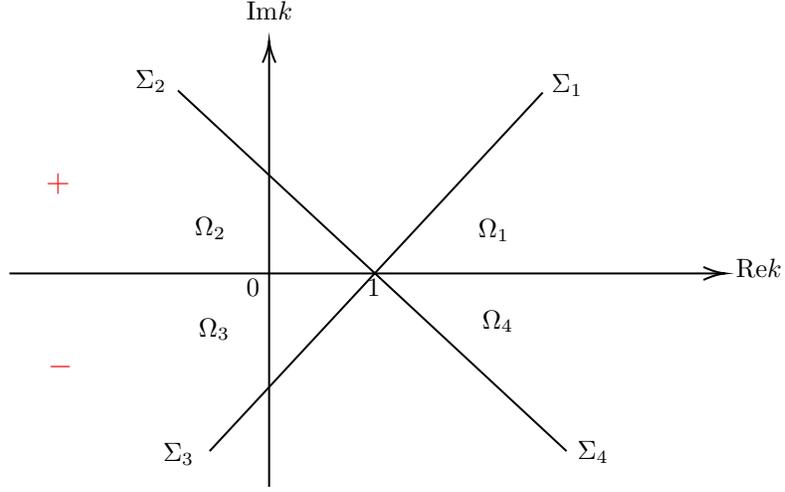

In this section,   we investigate the long-time asymptotic behavior in  the  region  $y/t>0$.
In this case, there is no stationary point, and the jump matrix $\widehat{V}(k)$ takes the following decomposition:
\begin{equation}
\widehat{V}(k)=\left(\begin{array}{cc}
1&-\bar{r}(k)e^{-2it\theta}\\0&1\end{array}\right)\left(\begin{array}{cc}
1&0\\r(k)e^{2it\theta}&1\end{array}\right).
\end{equation}

Considering the reconstruction formula of $u$, we choose to open the jump line at $k=1$. Define several regions and lines as shown in Figure \ref{line2}:
\begin{eqnarray}
&&\Omega_{2n+1}=\{k\in\mathbb{C}|n\pi\leq \arg(k-1)\leq n\pi+\frac{n\pi}{4}\},\\
&&\Omega_{2n+2}=\{k\in\mathbb{C}|(n+1)\pi-\frac{n\pi}{4}\leq \arg(k-1)\leq (n+1)\pi\},
\end{eqnarray}
where $n=0,1$;
\begin{equation}
\Sigma_j=e^{\frac{(j-1)i\pi}{2}+\frac{i\pi}{4}}(\Gamma_+),\quad\Gamma_+=\{k\in\mathbb{R}|k\geq 1\},\quad  j=1,2,3,4,
\end{equation}
and denote $\Sigma^{(1)}:=\bigcup_{j=1}^4\Sigma_j$.
\subsection{A mixed $\bar{\partial}$-RH problem and its decomposition}
In this part, we open the jump line at $k=1$ to make a continuous extension, and the first step is to introduce several new functions:
\begin{proposition}
Define functions $R_j\equiv R_j(k):\bar{\Omega}_j\rightarrow\mathbb{C}, j=1,2,3,4$ with the following boundary values:
\begin{eqnarray}
&&R_1=\left\{\begin{array}{ll}
-r(k),&k\in\Gamma_+,\\-r(1),& k\in \Sigma_1,\end{array}\right.,\qquad
R_2=\left\{\begin{array}{ll}
-r(k),& k\in\mathbb{R}\setminus\Gamma_+,\\-r(1),& k\in \Sigma_2,\end{array}\right.\\
&&R_3=\left\{\begin{array}{ll}
-\bar{r}(k),& k\in\mathbb{R}\setminus\Gamma_+,\\-\bar{r}(1),& k\in \Sigma_3,\end{array}\right.,\quad
R_4=\left\{\begin{array}{ll}
-\bar{r}(k),& k\in \Gamma_+,\\-\bar{r}(1),& k\in \Sigma_4,\end{array}\right.
\end{eqnarray}
And $R_j$ satisfies the following estimates:
\begin{equation}\label{r14}
|\bar{\partial}R_j|\lesssim |r'(\mathrm{Re}k)|+|k-1|^{-1/2},\ for\ k\in \Omega_j,\ j=1,4;
\end{equation}
\begin{equation}\label{r23}
|\bar{\partial}R_j|\lesssim |r'(\mathrm{Re}k)|+|k-1|^{-1/2}+\varphi(\mathrm{Re}k), for\ k\in \Omega_j,\ j=2,3,
\end{equation}
where $\varphi \in C^{\infty}_0(\mathbb{R},[0,1])$ is a fixed cut-off function with support near 0.

In a small fixed neighborhood of 0, $for\ k \in \Omega_j,\ j=2,3,\ $we have
\begin{equation}\label{r230}
|\bar{\partial}R_j|\lesssim |k|^2\rightarrow0,\ as\ k\rightarrow0.
\end{equation}
\end{proposition}
\begin{proof}
 We first give the proof of \eqref{r23} and take $R_2(k)$ for example. Here we take the cut-off function $\chi_0(x)\in\ C^{\infty}_0([0,1])$ as $\rho_0=1$:
\begin{equation}
\chi_0(x)=\left\{\begin{array}{ll}
1,&|x|\leqslant1/8,\\
0,&|x|\geqslant1/4.
\end{array}\right.
\end{equation}
Define the function $R_2$ as follows:
\begin{equation}
R_2=R_{21}+R_{22},
\end{equation}
where
\begin{equation}
\begin{aligned}
R_{21}=&\left(r(1)-[1-\chi_0(\mathrm{Re}k)]r(\mathrm{Re}k)\right)cos^2(2\alpha)-r(1),\\
R_{22}=&\tilde{f}(\mathrm{Re}k)cos^2(2\alpha)+\frac{i}{2}\rho e^{-i\alpha}sin(2\alpha)cos(2\alpha)\chi_0(\alpha)\tilde{f}'(\mathrm{Re}k)+sin^2(2\alpha)\chi_0(\alpha)\tilde{f}(\mathrm{Re}k)\\
&-\frac{1}{4}\rho e^{-i\alpha}sin^2(2\alpha)\tilde{f}'(\mathrm{Re}k)-\frac{1}{8}\rho^2 e^{-2i\alpha}sin^2(2\alpha)\tilde{f}''(\mathrm{Re}k),
\end{aligned}
\end{equation}
and
\begin{equation}
\tilde{f}(\mathrm{Re}k)=-\chi_0(\mathrm{Re}k)r(\mathrm{Re}k).
\end{equation}
Denote $k=1+\rho e^{i\alpha}$, then $\rho=|k-1|$ and
$
\bar{\partial}=\frac{1}{2}e^{i\alpha}(\partial_{\rho}+i\rho^{-1}\partial_{\alpha}).
$
So we have
\begin{equation}\label{dbar21}
\begin{aligned}
\bar{\partial}R_{21}=&\frac{1}{2}\chi_0'(\mathrm{Re}k)r(\mathrm{Re}k)cos^2(2\alpha)-2ie^{i\alpha}\rho^{-1}\chi_0(\mathrm{Re}k)r(\mathrm{Re}k)sin(2\alpha)cos(2\alpha)
\\&-\frac{1}{2}(1-\chi_0(\mathrm{Re}k))r'(\mathrm{Re}k)cos^2(2\alpha)+2ie^{i\alpha}\rho^{-1}(r(\mathrm{Re}k)-r(1))sin(2\alpha)cos(2\alpha).
\end{aligned}
\end{equation}
Considering that $r(\mathrm{Re}k),\ \rho^{-1}$ is bounded in the support of $\chi_0(\mathrm{Re}k)$, the first two terms of \eqref{dbar21} are  controlled by a function $\varphi\in C^{\infty}_0(\mathbb{R},[0,1])$. The last term of \eqref{dbar21} is estimated based on the inequality:
\begin{equation}
|r(\mathrm{Re}k)-r(1)|\lesssim|\mathrm{Re}k-1|^{1/2}\leqslant\rho^{1/2},
\end{equation}
and then
\begin{equation}
|\bar{\partial}R_{21}|\lesssim\varphi(\mathrm{Re}k)+|r'(\mathrm{Re}k)|+|k-1|^{-1/2}.
\end{equation}
As for $R_{22}$, it's obvious that it has the same formula as $R_{11,2}$ in the proof of Proposition \ref{Rjl}.
We can similarly get the following results:
\begin{equation}
|\bar{\partial}R_{11,2}|\lesssim\varphi(\mathrm{Re}k),\ k\in \Omega_{11},\quad
|\bar{\partial}R_{11,2}|\lesssim|k|^2,\quad as\ k\to 0.
\end{equation}
Combining all above, we derive \eqref{r23} and \eqref{r230}.

Finally, we give the proof of \eqref{r14}. Define $R_1$ by
\begin{equation}
R_1(k)=\left(r(1)-r(\mathrm{Re}k)\right)cos(2\alpha)-r(1),
\end{equation}
 and $\bar{\partial}R_1$ is estimated by the similar way as $\bar{\partial}R_{21}$.
\end{proof}
Now we introduce a matrix function by the following transformation
\begin{equation}
M^{(1)}(k)\equiv M^{(1)}(y,t,k)=\widehat{M}(k)R^{(1)}(k),
\end{equation}
where
\begin{equation}
R^{(1)}(k)=\left\{\begin{array}{ll}
\left(\begin{array}{cc}1&0\\R_je^{2it\theta}&1\end{array}\right),& k\in\Omega_1\cup \Omega_2;\\
\left(\begin{array}{cc}1&R_je^{-2it\theta}\\0&1\end{array}\right),& k\in\Omega_3\cup \Omega_4;\\
I,& elsewhere.\end{array}\right.
\end{equation}
Then $M^{(1)}(k)$ is the solution of a mixed $\bar{\partial}$-RH problem as follows:
\begin{rhp}Find a $2\times 2$ matrix-valued function $M^{(1)}(k)$ admitting the following properties:
\begin{itemize}
  \item Analyticity: $M^{(1)}(k)$ is continuous in $\mathbb{C}\setminus\Sigma^{(1)}$;
  \item Jump condition: $M^{(1)}(k)$ has continuous boundary values $M^{(1)}_{\pm}(k)$ on $\Sigma^{(1)}$, and
      \begin{equation}
      M^{(1)}_+(k)=M^{(1)}_-(k)V^{(1)}(k),\quad k\in \Sigma^{(1)},
      \end{equation}
      where
      \begin{equation}
      V^{(1)}(k)=\left\{\begin{array}{ll}
      (R^{(1)}(k))^{-1},&k\in\Sigma_1\cup\Sigma_2,\\
      R^{(1)}(k),&k\in\Sigma_3\cup\Sigma_3;
      \end{array}\right.
      \end{equation}
  \item Asymptotic behavior:
$
      M^{(1)}(k)\sim I+\mathcal{O}(k^{-1}),\quad as\ k \to \infty;
$
  \item $\bar{\partial}$- derivative:
  \begin{equation}
  \bar{\partial}M^{(1)}(k)=M^{(1)}(k)\bar{\partial}R^{(1)}(k).
  \end{equation}
\end{itemize}
\end{rhp}

To solve $M^{(1)}(k)$, we decompose it into a model RH problem for $M^{R}(k)\equiv M^{R}(y,t,k)$ with $\bar{\partial}M^{R}(k)\equiv0$ and a pure $\bar{\partial}$- Problem.
First we give a RH problem for $M^{R}(k)$:
\begin{rhp}\label{mr2}
Find a $2\times 2$ matrix-valued function $M^{R}(k)$ satisfying the properties as follows:
\begin{itemize}
  \item Analyticity: $M^{R}(k)$ is analytic in $\mathbb{C}\setminus\Sigma^{(1)}$;
  \item Jump condition: $M^R(k)$ has continuous boundary values $M^R_{\pm}(k)$ on $\Sigma^{(1)}$ and
      \begin{equation}
      M^{R}_+(k)=M^{R}_-(k)-V^{(1)}(k),\quad k\in \Sigma^{(1)};
      \end{equation}
  \item Asymptotic behavior:
$
  M^{R}(k)\sim I+\mathcal{O}(k^{-1}),\quad as\  k\to\infty.
$
\end{itemize}
\end{rhp}
Then, we define $M^{(2)}(k)\equiv M^{(2)}(y,t,k)$ as
\begin{equation}
M^{(2)}(k)= M^{(1)}(k)M^{R}(k)^{-1}.
\end{equation}
$M^{(2)}(k)$ is solution of a new $\bar{\partial}$-problem as follows:
\begin{dbar}Find a $2\times 2$ matrix-valued function $M^{(2)}(k)$ with the following properties:
\begin{itemize}
  \item Analyticity: $M^{(2)}(k)$ is continuous in $\mathbb{C}$.
  \item Asymptotic behavior:
$
  M^{(2)}(k)\rightarrow I,\quad k\rightarrow \infty;
$
  \item $\bar{\partial}$-Derivative:
  \begin{equation}
  \bar{\partial}M^{(2)}(k)=M^{(2)}(k)W(k),\quad k\in\mathbb{C},
  \end{equation}
  where
  \begin{equation}
  W(k)=M^{R}(k)\bar{\partial}R^{(1)}(k)(M^{R}(k))^{-1}.
  \end{equation}
\end{itemize}
\end{dbar}
\subsection{Analysis on the pure RH problem}
We first give an estimate for the jump matrix of $M^R(k)$ as follows:
\begin{proposition}\label{prop9}
As $t\rightarrow \infty$, we have
\begin{equation}
\|V^{(1)}(k)-I\|_{L^2(\Sigma^{(1)})}=O(t^{-1/2}) .
\end{equation}
\end{proposition}
\begin{proof}
Take the case in $\Sigma_1$ for example:
\begin{equation}
\| V^{(1)}(k)-I\|_{L^{2}(\Sigma_1)}=\left(\int_{\Sigma_1}|r(1)e^{2it\theta}|^2\mathrm{d}s\right)^{1/2}\lesssim \left(\int_0^{+\infty}(e^{-ctl})^2\mathrm{d}l\right)^{1/2}\lesssim t^{-1/2},
\end{equation}
where $s=1+le^{\frac{3i}{4}},\ \mathrm{Im}s=\frac{\sqrt{2}}{2}l.$
\end{proof}
To solve the RH problem $M^R(k)$, we introduce an operator $\mathcal{C}_{\omega}:\ L^2(\Sigma^{(1)})\to L^2(\Sigma^{(1)})$ as:
\begin{equation}
\mathcal{C}_{\omega}f=\mathcal{C}_-(f(V^{(1)}-I)),
\end{equation}
where $\mathcal{C}_-$ is the Cauchy projection operator on $\Sigma^{(1)}$.  
Based on Proposition \ref{prop9}, we have
\begin{equation}
\| \mathcal{C}_{\omega_e}\|_{L^2}\leqslant\| \mathcal{C}_-\|_{L^2}\| V^{(1)}-I\|_{L^2}\lesssim t^{-1/2},
\end{equation}
thus there exists $\mu(k)$ satisfying
\begin{equation}
(I-\mathcal{C}_{\omega_e})\mu(k)=I.
\end{equation}
Therefore, there exists a unique solution $M^R(k)$ for RH problem \ref{mr2} represented by
\begin{equation}
M^R(k)=I+\frac{1}{2\pi i}\int_{\Sigma^{(1)}}\frac{\mu(s)(V^{(1)}(s)-I)}{s-k}\mathrm{d}s
\end{equation}
with the following property:
\begin{proposition}\label{MR-2}
$M^R(k)$ has the asymptotic expansion as $k\to0$:
\begin{equation}
M^R(k)=M^R(0)+M^R_1k+M^R_2k^2+\mathcal{O}(k^3),
\end{equation}
where
\begin{equation}
M^R_1=\frac{1}{2\pi i}\int_{\Sigma^{(1)}}\frac{\mu(s)(V^{(1)}-I)}{s^2}\mathrm{d}s,\quad M^R_2=\frac{1}{2\pi i}\int_{\Sigma^{(1)}}\frac{\mu(s)(V^{(1)}-I)}{s^3}\mathrm{d}s.
\end{equation}
Besides, they satisfy the estimates as follows:
\begin{equation}
|M^R(0)-I|\lesssim t^{-1/2},\ |M^R_j|\lesssim t^{-1/2},\ j=1,2.
\end{equation}
\end{proposition}
\begin{proof}
Considering that $s^{-1}$ is bounded on $\Sigma^{(1)}$, the results are obvious, so we omit the proof.
\end{proof}
\subsection{Analysis on the pure $\bar{\partial}$-problem}
In this part, we deal with the pure $\bar{\partial}$-problem $M^{(2)}(k)$, which is presented by
\begin{equation}
M^{(2)}(k)=I+\frac{1}{\pi}\iint_{\mathbb{C}}
\frac{M^{(2)}(s)W(s)}{s-k}\mathrm{d}A(s),
\end{equation}
where $A(s)$ is the Lebesgue measure on $\mathbb{C}$. Denote the operator $S$ as:
\begin{equation}
S[f](k)=\frac{1}{\pi}\iint_{\mathbb{C}}
\frac{f(s)W(s)}{s-k}\mathrm{d}A(s),
\end{equation}
then
$
(I-S)M^{(2)}(k)=I.
$ We have an estimate of $S$ as follows:
\begin{proposition}
As $t\rightarrow \infty$, the norm of the integral operator S decays to zero:
\begin{equation}
\|S\|_{L^{\infty}\to L^{\infty}}\lesssim t^{-1/2},
\end{equation}
which implies that $(I-S)^{-1}$ exists.
\end{proposition}
\begin{proof}
The proof is similar as Proposition \ref{s-op}.
\end{proof}
Therefore, $M^{(2)}(k)$ exists uniquely and satisfies the following Proposition:
\begin{proposition}\label{M2-2}
As $k \rightarrow 0$, $M^{(2)}(k)$ has the following asymptotic expansion:
\begin{equation}
M^{(2)}(k)=M^{(2)}(0)+M^{(2)}_1k+M^{(2)}_2k^2+O(k^3),
\end{equation}
where
\begin{equation}
|M^{(2)}(0)-I|=|\frac{1}{\pi}\iint_{\mathbb{C}}
\frac{M^{(2)}(s)W(s)}{s}\mathrm{d}A(s)|\lesssim t^{-1/2},
\end{equation}
and
\begin{eqnarray}
&&M^{(2)}_1=-\frac{1}{\pi}\iint_{\mathbb{C}}
\frac{M^{(2)}(s)W(s)}{s^2}\mathrm{d}A(s),\\
&&M^{(2)}_2=-\frac{1}{\pi}\iint_{\mathbb{C}}
\frac{M^{(2)}(s)W(s)}{s^3}\mathrm{d}A(s)
\end{eqnarray}
satisfying
\begin{equation}
|M^{(2)}_j|\lesssim t^{-1/2},\ as\ t\rightarrow \infty,\ for\ j=1,2.
\end{equation}
\end{proposition}
\begin{proof}
The proof is similar as Proposition \ref{estiM3}.
\end{proof}
\subsection{Long-time asymptotics  for the HS equation}
Now we give the asymptotic expression of the solution $\widehat{M}(k)$  in the case $\xi>0$: as $k\to 0$,
\begin{equation}
\begin{aligned}
\widehat{M}(k)=&M^{(2)}(k)M^R(k)(R^{(1)})^{-1}(M^{(2)}(0)+M^{(2)}_1k+M^{(2)}_2k^2)(M^{R}(0)+M^{R}_1k+M^{R}_2k^2)+\mathcal{O}(k^3)\\
=&(I+\mathcal{O}(t^{-1/2}))(I+\mathcal{O}(t^{-1/2}))+\mathcal{O}(k^3)=I+\mathcal{O}(t^{-1/2})+\mathcal{O}(k^3),
\end{aligned}
\end{equation}
 which together with  the reconstruction formula \eqref{rstr-u}  gives the asymptotic behavior of the solution  $u(x,t)$ of the HS equation  by 
\begin{equation}
u(x,t)=u(y(x,t),t)=\mathcal{O}(t^{-1/2}),\quad t\to\infty,
\end{equation}
and
\begin{equation}
x(y,t)=y-\frac{1}{2i}
\widehat{M}_{1,1}(y,t)=y+\mathcal{O}(t^{-1/2}),\quad t\to\infty.
\end{equation}
In conclusion, we get the main result of this paper as shown in Theorem \ref{main}.
\vspace{5mm}

  \noindent\textbf{Acknowledgements}

    This work is supported by  the National Natural Science
    Foundation of China (Grant No. 11671095, 51879045).\vspace{2mm}

    \noindent\textbf{Data Availability Statements}

    The data that supports the findings of this study are available within the article.\vspace{2mm}

    \noindent{\bf Conflict of Interest}

    The authors have no conflicts to disclose.

\end{document}